\documentclass[12pt, reqno, a4paper]{amsart}

\usepackage{ amssymb, amsmath, enumerate, amsfonts, amsthm, mathrsfs, url, bm}

\numberwithin{equation}{section}

\usepackage{xcolor}  	
\usepackage[backref]{hyperref}
\hypersetup{
	colorlinks,
    linkcolor={blue!60!black},
    citecolor={blue!60!black},
    urlcolor={red!60!black}
}

\usepackage{color}

\usepackage[margin=1.25in]{geometry}

\RequirePackage{doi}

\usepackage[square,sort,comma,numbers]{natbib}
\newcommand{\C}{\mathbb{C}}

\newcommand{\N}{\mathbb{N}} 
\newcommand{\Z}{\mathbb{Z}}

\newcommand{\R}{\mathbb{R}}

\newcommand{\T}{\mathbb{T}}

\newcommand{\ve}{\varepsilon}

\newcommand{\f}[2]{\frac{#1}{#2}}

 \newcommand{\set}[1]{\left\{#1\right\}}

\newcommand{\biggset}[1]{\biggl\{ #1 \biggr\}}

\newcommand{\abs}[1]{\left| #1\right|}
\newcommand{\bigabs}[1]{\bigl| #1 \bigr|}

\newcommand{\biggabs}[1]{\biggl| #1 \biggr|}

\newcommand{\sqbrac}[1]{\left[ #1 \right]}

\newcommand{\ceil}[1]{\left\lceil #1 \right\rceil}

\newcommand{\floor}[1]{\left\lfloor #1 \right\rfloor}
\newcommand{\brac}[1]{\left( #1 \right)}
\newcommand{\bigbrac}[1]{\bigl( #1 \bigr)}
\newcommand{\Bigbrac}[1]{\Bigl( #1 \Bigr)}
\newcommand{\biggbrac}[1]{\biggl( #1 \biggr)}

\newcommand{\norm}[1]{\left\| #1\right\|}
\newcommand{\bignorm}[1]{\big\| #1 \big\|}

\newcommand{\ang}[1]{\left\langle#1\right\rangle}

\newcommand{\recip}[1]{\frac{1}{#1}}
\newcommand{\trecip}[1]{\tfrac{1}{#1}}

\newcommand{\vh}{\underline{h}}

\newcommand{\E}{\mathbb{E}}

\newcommand{\hcf}{\mathrm{hcf}}

\newcommand{\intd}{\mathrm{d}}
\newcommand{\supp}{\mathrm{supp}}

\newcommand{\eps}{\varepsilon}
\newcommand{\hash}{\#}

\makeatletter
\let\@@pmod\pmod
\DeclareRobustCommand{\pmod}{\@ifstar\@pmods\@@pmod}
\def\@pmods#1{\mkern4mu({\operator@font mod}\mkern 6mu#1)}
\makeatother

\newtheorem{theorem}{Theorem}[section]

\newtheorem{claim}[theorem]{Claim}

\newtheorem{lemma}[theorem]{Lemma}

\theoremstyle{definition}
\newtheorem{definition}[theorem]{Definition}
\newtheorem*{remark}{Remark}

\numberwithin{theorem}{section}

\title[Inverse nonlinear Roth]{The inverse theorem for the nonlinear Roth configuration: an exposition}

\author{Sean Prendiville}
\address{Department of Mathematics and Statistics\\ Lancaster University %\\ Manchester%\\ M13 9PL
\\ UK}
\email{s.prendiville@lancaster.ac.uk}

\begin{document}

\begin{abstract}
We give an exposition of the inverse theorem for the cut-norm associated to the nonlinear Roth configuration, established by Peluse and the author in \cite{PelusePrendivilleQuantitative}.
\end{abstract}

\maketitle

\setcounter{tocdepth}{1}
\tableofcontents

\maketitle

\section{Introduction}\label{introduction}

 Peluse and the author recently obtained an effective bound on the density of sets of integers lacking the configuration
\begin{equation}\label{main config}
x,\ x+y, \ x+y^2 \qquad (y \neq 0).
\end{equation}
We call this pattern the \emph{nonlinear Roth configuration}, after Bourgain and Chang \cite{BourgainChangNonlinear}.
\begin{theorem}[Peluse and Prendiville \cite{PelusePrendivilleQuantitative}]\label{main}
There exists an absolute constant $c > 0$ such that if $A \subset \set{1,2 ,\dots, N}$ lacks the configuration \eqref{main config}, then 
$$
|A| \ll N(\log\log N)^{-c}.
$$
\end{theorem}
We have since removed a logarithm from this bound.
\begin{theorem}[Peluse and Prendiville \cite{PelusePrendivillePolylogarithmic}]\label{main}
There exists an absolute constant $c > 0$ such that if $A \subset \set{1,2 ,\dots, N}$ lacks the configuration \eqref{main config}, then 
$$
|A| \ll N(\log N)^{-c}.
$$
\end{theorem}

The main innovation behind both of these results is \cite[Theorem 7.1]{PelusePrendivilleQuantitative}, an inverse theorem for the counting operator associated to this configuration. It is the purpose of this note to give an exposition of this inverse theorem. The approach is essentially the same as that in \cite{PelusePrendivilleQuantitative}. We hope that having two distinct accounts is useful for those  interested in utilising these ideas.
\begin{definition}[Counting operator]
For positive integers $q \leq N$ write
\begin{equation}\label{M def}
M:= \floor{\sqrt{N/q}}.
\end{equation}
Given this, define the \emph{counting operator} on the functions $f_i : \Z \to \C$  by
\begin{equation}\label{counting op}
\Lambda_{q, N}(f_0, f_1, f_2) := \E_{x \in [N]} \E_{y \in [M]} f_0(x)f_1(x+y) f_2(x+qy^2).
\end{equation}
When the $f_i$ all equal $f$ we simply write $\Lambda_{q, N}(f)$.
\end{definition}

\begin{definition}[Local function]\label{local function def}
We call a function $\phi : \Z \to \C$ a \emph{local function of resolution $M$ and modulus $q$} if there exists a partition of $\R$ into intervals of length $M$ such that $\phi$ is constant on the intersection of every such interval with every congruence class mod $q$.
\end{definition}

\begin{definition}[Cut norm]
Define the \emph{cut norm} of  $f : \Z \to \C$ by
\begin{equation}\label{q norm eq}
\norm{f}_{q, N} := \sup\{|\Lambda_{q, N}(f, g_1,g_2)|,\ |\Lambda_{q, N}(g_1, f,g_2)|,\ |\Lambda_{q, N}(g_1,g_2, f)|\},
\end{equation}
where the supremum is taken over all 1-bounded functions $g_i : [N] \to \C$.  We note that, in spite of our nomenclature, this is not a norm but a seminorm.  One could remedy this by summing over $y \geq 0$ in the counting operator \eqref{counting op}

This seminorm is useful in \cite{PelusePrendivillePolylogarithmic}. However, it is too restrictive for the approach developed in \cite{PelusePrendivilleQuantitative}, where we (implicitly) only work with the following quantities:
\begin{equation}\label{sharp q norm eq}
\norm{f}^\sharp_{q} := \sup\{ |\Lambda_{q, N}(g_0,g_1, f)|:  |g_i| \leq 1\ \text{ and }\ \supp(g_i) \subset [N] \}
\end{equation}
and
\begin{equation}\label{flat q norm eq}
\norm{f}^\flat_{q} := \sup\{|\Lambda_{q, N}(f, g_1,g_2)|,\ |\Lambda_{q, N}(g_1, f,g_2)|\ : \ |g_i| \leq 1\ \text{ and }\ \supp(g_i) \subset [N] \}.
\end{equation}
\end{definition}

Here then is a re-formulation and slight generalisation of \cite[Theorem 7.1]{PelusePrendivilleQuantitative}.

\begin{theorem}[Partial cut norm inverse theorem]\label{partial inverse theorem}
Let $q \leq N$ be positive integers, $\delta>0$, and $f:\Z\to\C$ be a $1$-bounded function with support in $[N]$.  
Suppose that  
\[
\norm{f}^\flat_{q,N}\geq\delta .
\]
Then either $N \ll (q/\delta)^{O(1)}$ or there exists a 1-bounded local function $\phi$ of resolution  $\gg (\delta/q)^{O(1)}N^{1/2}$, modulus $qq'$ for some $q'\ll  \delta^{-O(1)}$, and  such that
$$
\sum_{x\in[N]} f(x)\phi(x) \gg \delta^{2^{66}} N.
$$
\end{theorem}

This exposition is organised as follows. In \S\ref{sec3}, we give a more detailed outline of the proof of Theorem~\ref{partial inverse theorem}. In \S\S\ref{pet}--\ref{concat sec} we develop an effective approach to a (special case of a) so-called \emph{concatenation} theorem of Tao and Ziegler \cite{TaoZieglerConcatenation}.  This allows us to show that if our counting operator is large, then the function weighting the nonlinear term must have large Gowers uniformity norm. The drawback  is that the degree of the resulting Gowers norm is large (in our approach it is the $U^5$-norm).  In \S\ref{sec6} we give a \emph{degree-lowering} procedure, which utilises properties specific to our configuration to show that one may replace the $U^5$-norm with the $U^1$-norm. In \S\ref{inverse theorem section} we combine the results of the previous sections in order to prove Theorem \ref{partial inverse theorem}.

\subsection{Notation} 
\subsubsection{Standard conventions}
We use $\N$ to denote the positive integers.  For a real $X \geq 1$, write $[X] = \{ 1,2, \ldots, \floor{X}\}$.  A complex-valued function is \emph{1-bounded} if the modulus of the function does not exceed 1.

We use counting measure on $\Z$, so that for $f,g :\Z \to \C$ we have
$$
\ang{f,g} := \sum_x f(x)\overline{g(x)}\qquad \text{and}\qquad \norm{f}_{L^p} := \biggbrac{\sum_x |f(x)|^p}^{\recip{p}}.
$$ 
Any sum of the form $\sum_x$ is to be interpreted as a sum over $\Z$. 
We use Haar probability measure on $\T := \R/\Z$, so that for measurable $F : \T \to \C$ we have
$$
\norm{F}_{L^p} := \biggbrac{\int_\T |F(\alpha)|^p\intd\alpha}^{\recip{p}} = \biggbrac{\int_0^1 |F(\alpha)|^p\intd\alpha}^{\recip{p}}
$$
For $\alpha \in \T$ we write $\norm{\alpha}$ for the distance to the nearest integer.  %This induces a metric on $\T$ via $\norm{\alpha - \beta}$.

For a finite set $S$ and function $f:S\to\C$, denote the average of $f$ over $S$ by
\[
\E_{s\in S}f(s):=\f{1}{|S|}\sum_{s\in S}f(s).
\]

Given functions $f,g : G \to \C$ on an additive group with measure $\mu_G$ we define their convolution by 
\begin{equation}\label{convolution}
f*g(x) := \int_G f(x-y) g(y) \intd\mu_G,
\end{equation}
when this makes sense.

We define the Fourier transform of $f : \Z \to \C$ by 
\begin{equation}\label{Fourier transform}
\hat{f}(\alpha) := \sum_x f(x) e(\alpha x) \qquad (\alpha \in \T),
\end{equation}
again, when this makes sense.  Here $e(\alpha)$ stands for $e^{2\pi i \alpha}$.

The difference function of $f : \Z \to \C$ is the function $\Delta_h f : \Z \to \C$ given by 
$$
\Delta_hf(x) = f(x) \overline{f(x+h)}.
$$
Iterating gives
$$
\Delta_{h_1, \dots, h_s} f := \Delta_{h_1} \dots \Delta_{h_s} f.
$$
This allows us to define the Gowers $U^s$-norm
\begin{equation}\label{Us def}
\norm{f}_{U^s} := \brac{\sum_{x, h_1, \dots, h_s} \Delta_{h_1, \dots, h_s} f(x)}^{1/2^s}.
\end{equation}

If $\|\cdot\|$ is a seminorm on an inner product space, recall that its dual seminorm $\|\cdot\|^*$ is defined by
\[
\|f\|^{*}:=\sup_{\|g\|\leq1}|\langle f,g\rangle|.
\]
Hence
\begin{equation}\label{dual ineq}
\abs{\ang{f,g}} \leq \norm{f}^* \norm{g}.
\end{equation}

For a function $f$ and positive-valued function $g$, write $f \ll g$ or $f = O(g)$ if there exists a constant $C$ such that $|f(x)| \le C g(x)$ for all $x$. We write $f = \Omega(g)$ if $f \gg g$.  We sometimes opt for a more explicit approach, using $C$ to denote a large absolute constant, and $c$ to denote a small positive absolute constant.  The values of $C$ and $c$ may change from line to line. 
\subsubsection{Local conventions}
Up to normalisation, all of the above are well-used in the literature. Next we list notation specific to our paper. We have tried to minimise this in order to aid the casual reader.  

For a real parameter $H \geq 1$, we use $\mu_H : \Z \to [0,1]$ to represent the following normalised Fej\'er kernel
\begin{equation}\label{fejer}
\mu_H(h) := \recip{\floor{H}} \brac{1 - \frac{|h|}{\floor{H}}}_+ = \frac{(1_{[H]} * 1_{[H]} )(h)}{\floor{H}^2}.
\end{equation}
For a multidimensional vector $h \in \Z^d$ we write
\begin{equation}\label{multidim fejer}
\mu_H(h) := \mu_H(h_1)\dotsm \mu_H(h_d).
\end{equation}
We observe that this is a probability measure on $\Z^d$ with support in the interval $(-H, H)^d$.

\section{An outline of our argument}\label{sec3}

In this section we describe the ideas behind Theorem \ref{partial inverse theorem}. In the hope of making the ideas clearer, we make the simplification that $q=1$ in our counting operator \eqref{counting op}.  Hence, for finitely supported functions $f_0,f_1,f_2:\Z\to\C$, write
\begin{equation}\label{eq3.3}
\Lambda(f_0,f_1,f_2):=\E_{x\in [N]}\E_{y\in[N^{1/2}]}f_0(x)f_1(x+y)f_2(x+y^2).
\end{equation}
For this operator, Theorem \ref{partial inverse theorem} can be deduced from the following.
\begin{lemma}\label{lem3.1}
Let $f_0,f_1,f_2:\Z\to\C$ be $1$-bounded functions supported in the interval $[N]$ and $\delta>0$. Suppose that
\[
|\Lambda(f_0,f_1,f_2)|\geq\delta.
\]
Then either $N\ll \delta^{-O(1)}$ or there exist positive integers $q\ll\delta^{-O(1)}$ and $ N'\gg\delta^{O(1)}N^{1/2}$ such that
\begin{equation}\label{eq3.4}
\sum_x\left|\sum_{y\in[N']}f_1(x+qy)\right|\gg\delta^{O(1)}NN'.
\end{equation}
\end{lemma}
\noindent Using the notation \eqref{Us def}, notice that the left-hand side of \eqref{eq3.4} is equal to 
$$
\sum_x \norm{f_1}_{U^1(x + q\cdot [N'])}.
$$

\subsection{Quantitative concatenation}\label{ss3.2}

To prove Lemma~\ref{lem3.1}, we first prove that our counting operator~\eqref{eq3.3} is controlled by the $U^5$-norm of $f_2$. The purpose of this subsection is to sketch how we do this with polynomial bounds.

By applying the Cauchy--Schwarz and van der Corput inequalities a number of times, we show in \S\ref{pet} that, when $f_0,f_1,f_2:\Z\to\C$ are $1$-bounded functions supported in the interval $[N]$, largeness of the counting operator~\eqref{eq3.3} implies largeness of the sum
\begin{equation}\label{eq3.5}
\sum_{a,b\in[N^{1/2}]}\sum_{h_1,h_2,h_3\in[N^{1/2}]}\sum_x\Delta_{ah_1,bh_2,(a+b)h_3}f_2(x).
\end{equation}
This deduction is made following the PET induction scheme of Bergelson and Leibman \cite{BergelsonLeibmanPolynomial}. The gain in working with the counting operator \eqref{eq3.5} over \eqref{eq3.3} is that univariate polynomials such as $y^2$, whose image constitute a sparse set, have been replaced by bilinear forms such as $ah_1$, whose image is much denser

In \S\S\ref{arithmetic inverse}--\ref{concat sec}, we show that largeness of~\eqref{eq3.5} implies largeness of $\|f_2\|_{U^5}$.  If there were no dependence between the coefficients of the $h_i$ in~\eqref{eq3.5}, then  we could in fact bound~\eqref{eq3.5} in terms of $\|f_2\|_{U^3}$. 
Since the argument is informative, we illustrate why this is the case for the sum
\begin{equation}\label{eq3.6}
\sum_{a,b,c\in[N^{1/2}]}\sum_{h_1,h_2,h_3\in[N^{1/2}]}\sum_x\Delta_{ah_1,bh_2,ch_3}f_2(x).
\end{equation}
The following fact is key, the formal version of which is Lemma \ref{densifying difference functions}. 
\begin{claim}\label{densify binary}
If $\displaystyle \sum_{a, h\in [N^{1/2}]} \sum_x \Delta_{ah} f(x)$ is large then so is $\displaystyle \sum_{k\in (-N, N)} \sum_x \Delta_{k} f(x)$.
\end{claim}
\begin{proof}[Sketch proof]
 Apply the Cauchy--Schwarz inequality to double the $a$ and $h$ variables, yielding a bound in terms of 
\begin{equation}\label{eq3.7}
\sum_{a,a'\in[N^{1/2}]}\sum_{h,h'\in[N^{1/2}]}\sum_x\Delta_{ah-a'h'}f(x).
\end{equation}
For a random choice of $a,a'\in[N^{1/2}]$, the progression $a\cdot[N^{1/2}]-a'\cdot[N^{1/2}]$ covers a large portion of the interval $(-N,N)$ relatively smoothly. One can make this intuition rigorous and thus deduce largeness of the sum
$
\sum_{k \in (-N, N)}\sum_x\Delta_{k}f(x).
$\end{proof}
Applying Claim \ref{densify binary} three times allows us to replace each of $ah_1$, $bh_2$  and $ch_3 $ in \eqref{eq3.6} with $k_1, k_2, k_3 \in (-N, N)$, yielding largeness of $\norm{f_2}_{U^3}$.  

Since the PET induction scheme outputs \eqref{eq3.5}, and not \eqref{eq3.6}, the problem remains of how to handle the dependency between the differencing parameters in \eqref{eq3.5}. If we were not concerned with  quantitative bounds, we could apply a `concatenation' theorem of Tao and Ziegler~\cite[Theorem 1.24]{TaoZieglerConcatenation} to obtain largeness of the $U^9$-norm of $f_2$.  However, the qualitative nature of this argument means that it cannot be used to obtain bounds in the nonlinear Roth theorem. In its place we prove Theorem~\ref{global U5}, which is a special case of \cite[Theorem 1.24]{TaoZieglerConcatenation}, using a very different argument that gives polynomial bounds. We spend the remainder of this subsection sketching the argument.

We begin by viewing  \eqref{eq3.5} as the average
\begin{equation}\label{eq3.9}
\sum_{a, h_1 \in [N^{1/2}]} \norm{\Delta_{ah_1} f_2}_a,
\end{equation} 
where
\begin{equation}\label{eq3.8}
\|f\|_a^4:=\sum_{b\in[N^{1/2}]}\sum_{h_2,h_3\in[N^{1/2}]}\sum_x\Delta_{bh_2,(a+b)h_3}f(x)
\end{equation}
One can view this as an average of 2-dimensional Gowers box norms where, for fixed $b$, the inner sum corresponds to a box norm in the `directions' $b$ and $a+b$.
Note that if we could bound the quantity $\|\Delta_{a h_1}f_2\|_a$ in terms of the $U^4$-norm of $\Delta_{a h_1}f_2$ for many pairs $(a,h_1)$, then by Claim \ref{densify binary} we deduce largeness of the $U^5$-norm of $f_2$. We show that, on average, one can indeed control $\|\cdot\|_a$ in terms of $\|\cdot\|_{U^4}$, with polynomial bounds. The following can be extracted from the proof of (the more general) Theorem \ref{global U5}.
\begin{lemma}\label{lem3.4}
For each $a\in[N^{1/2}]$ let $f_a:\Z\to\C$ be a $1$-bounded function supported in the interval $[N]$. Suppose that
\[
\E_{a\in[N^{1/2}]}\|f_a\|_a^4\geq \delta \norm{1_{[N]}}_a^4.
\]
Then
\[
\E_{a\in[N^{1/2}]}\|f_a\|_{U^4}^{16}\gg \delta^{O(1)}\norm{1_{[N]}}_{U^4}^{16}.
\]
\end{lemma}

To finish this subsection, we briefly discuss the proof of this key lemma. For most choices of $a, b\in[N^{1/2}]$, the `directions' $a$ and $a+b$ of the box norm 
\begin{equation}\label{box norm}
\sum_{h_2,h_3\in[N^{1/2}]}\sum_x\Delta_{bh_2,(a+b)h_3}f_a(x)
\end{equation}
are close to `independent', in the sense that at least one of the directions $a$ and $a+b$ is large and together they have small greatest common divisor. The proof of Lemma~\ref{lem3.4} thus begins by viewing $\|\cdot\|_a$ as an average of box norms
\begin{equation}\label{real box norm}
\|f\|_{\square(X,Y)}^4:=\sum_{x_1,x_2\in X, y_1,y_2\in Y}f(x_1,y_1)\overline{f(x_1,y_2)f(x_2,y_1)}f(x_2,y_2).
\end{equation}
It is easy to show that largeness of $\|f\|_{\square(X,Y)}$ implies that $f$ correlates with a function of the form $(x,y)\mapsto l(x)r(y)$. We show, analogously, that provided $b$ and $a+b$ are not too small and have  greatest common divisor not too large, then largeness of the arithmetic box norm \eqref{box norm} implies that $f_a$ correlates with a product  $g_bh_{a+b}$ of 1-bounded functions, where $g_b$ is $b$-periodic and $h_{a+b}$ is almost periodic under shifts by integer multiples of $a+b$. As a consequence, for most $a\in[N^{1/2}]$, largeness of $\|f_a\|_{a}$ implies largeness of
\begin{equation}\label{eq3.10}
\sum_{b\in[N^{1/2}]}\sum_xf_a(x)g_b(x)h_{a+b}(x).
\end{equation}
In fact, an application of Cauchy--Schwarz allows us give an explicit description of $h_{a+b}$ in terms of $f_a$, namely we may take it to be of the form
\begin{equation}\label{h description}
h_{a+b}(x) = \E_{k\in [N^{1/2}]} f_a(x+(a+b)k)g_b(x+(a+b)k).
\end{equation}
This presentation makes apparent the almost periodicity of $h_{a+b}$.

\begin{claim}\label{hab claim}
Largeness of \eqref{eq3.10} implies that $\E_{b \in [N^{1/2}]}h_{a+b}$ has large $U^3$-norm.
\end{claim}

Let us first show why Claim \ref{hab claim} in turn implies that $f_a$ has large $U^4$-norm, completing our sketch proof of Lemma \ref{lem3.4}.  The expression \eqref{h description} and the triangle inequality for Gowers norms together imply that largeness of $\E_{b \in [N^{1/2}]}\norm{h_{a+b}}_{U^3}$ implies largeness of $\E_{b \in [N^{1/2}]}\norm{f_ag_b}_{U^3}$.  Utilising the $b$-periodicity of $g_b$ we have
\begin{equation}\label{fg product}
\norm{f_ag_b}_{U^3} = \E_{k \in [N^{1/2}]} \norm{f_a(\cdot)g_b(\cdot + bk)}_{U^3}.
\end{equation}
The product $f_a(\cdot)g_b(\cdot + bk)$ resembles a difference function in the direction $b$. Indeed the Gowers--Cauchy--Schwarz inequality (see \cite[Exercise 1.3.19]{TaoHigher}) shows that if \eqref{fg product} is large (on average over $b \in [N^{1/2}]$) then so is
$$
\E_{b,k \in [N^{1/2}]} \norm{\Delta_{bk}f_a}_{U^3}
$$
Largeness of $\norm{f_a}_{U^4}$ then follows from Claim \ref{densify binary}.

Finally we sketch the proof of Claim \ref{hab claim}.  The Cauchy--Schwarz inequality allows us to remove the weight $f_a(x)$ from \eqref{eq3.10} and deduce largeness of
$$
\sum_x\sum_{b,b' \in [N^{1/2}]} \overline{g_b(x)h_{a+b}(x)}g_{b'}(x)h_{a+b'}(x).
$$
Using the periodicity properties of $g_b$, $g_{b'}$ and $h_{a+b}$, this is approximately equal to
\begin{equation*}%\label{eq3.10}
\sum_x\sum_{\substack{b,b' \in [N^{1/2}]\\k_1, k_2, k_3 \in [N^{1/2}]}}\overline{g_b(x-bk_1)h_{a+b}(x-(a+b)k_2)}g_{b'}(x-b'k_3)h_{a+b'}(x).
\end{equation*}
Changing variables in $x$, we obtain largeness of the sum
\begin{multline*}%\label{eq3.10}
\sum_x\sum_{\substack{b,b' \in [N^{1/2}]\\k_1, k_2,k_3 \in [N^{1/2}]}}\overline{g_b(x+(a+b)k_2+b'k_3)h_{a+b}(x+bk_1+b'k_3)}\\
g_{b'}(x+bk_1+(a+b)k_2)h_{a+b'}(x+bk_1+(a+b)k_2+b'k_3).
\end{multline*}
The point here is that all but the last function have arguments depending on at most two of the bilinear forms $bk_1$, $(a+b)k_2$ and $b'k_1'$. This enables us to employ the Gowers--Cauchy--Schwarz inequality (in the form of Lemma \ref{box cauchy}) to deduce largeness of a sum similar to
$$
\sum_x\sum_{\substack{b,b'\in [N^{1/2}]\\ k_1, k_2, k_3\in [N^{1/2}]}}\Delta_{bk_1,\, (a+b)k_2,\, b'k_3}h_{a+b'}(x).
$$
The utility of this expression is that the directions of the differencing parameters are all `independent' of the direction of periodicity of $h_{a+b'}$.  Indeed the approximate $(a+b')$-periodicity of $h_{a+b'}$ means that one can replace $\Delta_y h_{a+b'}$ with $\E_k\Delta_{y + (a+b')k} h_{a+b'}$ at the cost of a small error. We thereby obtain largeness of
\begin{equation}\label{horrendous differences}
\sum_x\sum_{b, b' \in [N^{1/2}]}\sum_{\substack{k_1, k_2, k_3\in [N^{1/2}]\\ k_1', k_2', k_3'\in [N^{1/2}]}}\Delta_{bk_1 +(a+b')k_1',\, (a+b)k_2+ (a+b')k_2',\, b'k_3+(a+b')k_3'}h_{a+b'}(x).
\end{equation}
For a random triple $(a,b,b') \in [N^{1/2}]$ the greatest common divisor of the pairs $(b, a+b')$, $(a+b, a+b')$ and  $(b', a+b')$ are all small, and these are the pairs appearing in the differencing parameters of \eqref{horrendous differences}. The argument used to treat \eqref{eq3.7} may be therefore be employed  to replace \eqref{horrendous differences} with
$$
\sum_x\sum_{b'\in [N^{1/2}]}\sum_{k_1, k_2, k_3\in [N]}\Delta_{k_1 ,k_2,k_3}h_{a+b'}(x),
$$
and thereby yield Claim \ref{hab claim}.

\subsection{Degree lowering}\label{ss3.3}

After we have shown that $\Lambda(f_0,f_1,f_2)$ is controlled by the $U^5$-norm of $f_2$, we carry out a `degree lowering' argument. This technique originated in the work~\cite{PelusePolynomial} in finite fields. The basic idea is that, under certain conditions, one can combine $U^s$-control with understanding of two-term progressions to deduce $U^{s-1}$-control. Repeating this gives a sequence of implications
\[
U^{5}\text{-control}\implies U^{4}\text{-control}\implies U^{3}\text{-control}\implies U^{2}\text{-control}\implies U^{1}\text{-control}.
\]
Despite the appearance of the $U^5$-norm, $U^4$-norm, and $U^3$-norm, the degree lowering argument, both in~\cite{PelusePolynomial} and here, does not require the $U^s$-inverse theorem for any $s\geq 3$. Instead it relies on Fourier analysis in the place of these inverse theorems.

Adapting the degree lowering argument of~\cite{PelusePolynomial} to the integer setting requires several significant modifications. The first modification is that the $U^s$-control described above is control in terms of the $U^s$-norm of the dual function
\begin{equation}\label{sketch dual}
F(x):=\E_{y\in[N^{1/2}]}f_0(x-y^2)f_1(x+y-y^2).
\end{equation}
Thus, to begin the degree lowering argument, we must show that largeness of $\Lambda(f_0,f_1,f_2)$ implies largeness of $\|F\|_{U^5}$. To do this, we use a simple Hahn--Banach decomposition as described in \cite[Proposition 3.6]{GowersDecompositions}, for details see \S\ref{inverse theorem section}. 

We conclude this section by sketching an instance of degree-lowering: how $U^3$-control of the dual \eqref{sketch dual} implies $U^2$-control, starting from the assumption that
\[
\|F\|_{U^3}^8\geq\delta \norm{1_{[N]}}_{U^3}^8.
\]
Using the fact that $\|F\|_{U^3}^8=\sum_h\|\Delta_hF\|_{U^2}^4$ and applying the $U^2$-inverse theorem, we deduce the existence of a function $\phi:\Z\to\T$ such that, for at least $\gg\delta N$ choices of differencing parameter $h$, we have
\begin{equation}\label{eq3.11}
\left|\sum_{x\in[N]}\Delta_hF(x)e(\phi(h)x)\right|\gg\delta N.
\end{equation}
Note that if, in the above inequality, we could replace the function $\phi(h)$ by a constant $\beta\in\T$ not depending on $h$, then we could easily deduce largeness of $\|F\|_{U^2}$. Indeed, writing $g(h)$ for the phase of the sum inside absolute values, this would give 
$$
\sum_{x, h}\overline{g(h)}\overline{F(x+h)}F(x)e(\beta x)\gg\delta^{O(1)}N^3,
$$
and the usual argument\footnote{One can either use orthogonality and extraction of a large Fourier coefficient, as in the proof of Lemma \ref{U2 inverse}, or use two applications of Cauchy--Schwarz.} showing $U^2$-control of the equation $x+y = z$ implies that $\|F\|_{U^2}^4\gg\delta^{O(1)}\norm{1_{[N]}}_{U^2}$. It thus remains to show that such a $\beta$ exists.

Expanding the definition of the difference and dual functions in \eqref{eq3.11}, and using the Cauchy--Schwarz inequality (as is done in greater generality in the proof of Lemma~\ref{dual difference interchange}), one can show that there exists $h'$ such that for many $h$ satisfying \eqref{eq3.11} we have
\[
\left|\sum_x\sum_{y\in[N^{1/2}]}\Delta_{h-h'}f_0(x)\Delta_{h-h'}f_1(x+y)e([\phi(h)-\phi(h')][x+y^2])\right|\gg \delta^{O(1)}N^{3/2}
\]
Further application of Cauchy--Schwarz allows us to remove the difference functions from the above inequality and deduce largeness of the exponential sum
\[
\sum_{z\in[N^{1/2}]}\left|\sum_{y\in[N^{1/2}]}e(2\sqbrac{\phi(h)-\phi(h')} yz)\right|.
\]
Summing the inner geometric progression and using a Vinogradov-type lemma then shows that $\phi(h)-\phi(h')$ is major arc. There are very few major arcs, so the pigeonhole principle gives the existence of $\beta_0\in\T$ such that $\phi(h)-\phi(h')$ is very close to $\beta_0$ for many $h\in(-N,N)$ that also satisfy~\eqref{eq3.11}. We may therefore take $\beta=\beta_0+\phi(h')$ in the argument following \eqref{eq3.11}.

\section{PET induction}\label{pet}
We prove Theorem \ref{partial inverse theorem} over the course of \S\S\ref{pet}--\ref{inverse theorem section}. We begin in \S\S\ref{pet}--\ref{concat sec} by showing how our counting operator $\Lambda_{q, N}(f_0, f_1, f_2)$, as defined in \eqref{counting op}, is controlled by the $U^5$-norm of $f_2$.  This argument starts with the PET induction scheme of Bergelson--Leibman \cite{BergelsonLeibmanPolynomial}, which in some sense `linearises' a polynomial progression, replacing univariate polynomials such as $y^2$ with bilinear forms $ah$.  The outcome of this procedure is Lemma \ref{linearisation}.

For the following, we recall our definition \eqref{fejer} of $\mu_H$.
\begin{lemma}[van der Corput inequality]\label{vdc}
Let $f:\Z\to\C$ be 1-bounded and $M, H \geq 1$. Then we have the estimate
\[
\biggabs{\E_{y\in [M]}f(y)}^2\leq \frac{M+H}{M}\sum_h\mu_H(h)\E_{y\in [M]}\Delta_h f(y) .
\]
\end{lemma}

\begin{proof}
This is standard, see for instance \cite[Lemma 3.1]{PrendivilleQuantitative}.\end{proof}

\begin{lemma}[Difference functions control linear configurations]\label{difference control}
 Let $f_{i}:\Z\to\C$ be $1$-bounded functions with support in an interval $I_i$ of size $|I_i| =N$.   
Then for any $a,b \in \Z$ and $1 \leq H\leq M$ we have
\begin{multline}\label{difference control ineq}
\biggabs{\E_{x\in I_0}\E_{y\in[M]}f_{0}(x)f_{1}(x+ay)f_{2}(x+by)f_{3}(x+(a+b)y)}^8\\
 \ll \sum_h\mu_{H}(h)\E_{x\in I_3}\Delta_{ah_1,bh_2,(a+b)h_3}f_{3}(x).
\end{multline}
\end{lemma}

\begin{proof}
Applying Cauchy-Schwarz in the $x$ variable gives
\begin{multline*}
\biggabs{\E_{x\in I_0}\E_{y\in[M]}f_{0}(x)f_{1}(x+ay)f_{2}(x+by)f_{3}(x+(a+b)y)}^2\\
 \leq
\recip{N}\sum_{x}\bigg|\E_{y\in[M]}f_1(x+ay)f_2(x+by)
f_3(x+(a+b)y)\bigg|^2.
\end{multline*}
 Bounding the inner sum using van der Corput's inequality (Lemma \ref{vdc}) and making the change of variables $x \mapsto x-ay$ (valid since $x$ is ranging over $\Z$), the latter is at most  
$$
2\sum_{h_1}\mu_H(h_1)\E_{x \in I_1}\E_{y \in [M]}\Delta_{ah_1}f_1(x) \Delta_{bh_1} f_2(x+(b-a)y) \Delta_{(a+b)h_1}f_3(x+by).
$$
Here we may restrict $x$ to $I_1$ on observing that the support of $\Delta_{ah_1} f_1$ is contained in the support of $f_1$.
Making use of the fact that $\mu_H$ is a probability measure, we repeat the procedure of applying Cauchy--Schwarz, van der Corput then a change of variables, to deduce that 
\begin{multline*}
\biggabs{\E_{x\in I_0}\E_{y\in[M]}f_{0}(x)f_{1}(x+ay)f_{2}(x+by)f_{3}(x+(a+b)y)}^4\\
 \leq 8 \sum_{h_1, h_2}\mu_H(h_1)\mu_H(h_2)\E_{x \in I_2}\E_{y\in[M]} \Delta_{bh_1, (b-a)h_2} f_2(x) \Delta_{(a+b)h_1, bh_2}f_3(x+ay).
\end{multline*}
A final iteration of the same procedure then yields \eqref{difference control ineq}.
\end{proof}

Before embarking on the following, we remind the reader of our convention \eqref{M def} regarding $M$.
\begin{lemma}[Linearisation]\label{linearisation}
Let $f_i:\Z\to\C$ be $1$-bounded functions, each with support in the interval $[N]$. Then for any  $1 \leq H\leq  M$ we have
\begin{equation}\label{linearised ineq}
\abs{\Lambda_{q, N}(f_0, f_1, f_2) }^{32} \ll
 \sum_{a,b, h}\mu_M(a)\mu_M(b)\mu_H(h)\E_{x\in [N]} \Delta_{2q(a+b)h_1,\, 2qbh_2,\, 2qah_3}f_2(x).
\end{equation}
\end{lemma}
\begin{proof}
We repeat the procedure given in the proof of Lemma \ref{difference control}, applying Cauchy-Schwarz, followed by van der Corput's inequality and a change of variables.  A first application gives
\begin{multline*}
\abs{\Lambda_{q, N}(f_0, f_1, f_2) }^2 \leq\\
 2 \sum_{a}\mu_M(a) \E_{x\in [N]} \E_{y\in [M]} \Delta_{a}f_1(x)f_2\bigbrac{x+qy^2-y}\overline{f_2\bigbrac{x+q(y+a)^2-y}}.
\end{multline*}
A second application then gives
\begin{multline*}
\abs{\Lambda_{q, N}(f_0, f_1, f_2) }^4 \ll
 \sum_{a,b}\mu_M(a)\mu_M(b) \E_{x\in [N]} \E_{y\in [M]} f_2(x)\overline{f_2\bigbrac{x+2qay + qa^2}}\\ \overline{f_2\bigbrac{x+2qby + qb^2-b}}f_2\bigbrac{x+2q(a+b)y + q(a+b)^2-b}.
\end{multline*}
Applying Lemma \ref{difference control} to bound the inner sum over $x$ and $y$, we obtain \eqref{linearised ineq} after a final change of variables 
\end{proof}

\section{An inverse theorem for the arithmetic box norm}\label{arithmetic inverse}

The objective in this section is to characterise those 1-bounded functions $f : \Z \to \C$ with support in $[N]$ for which the following quantity is large
\begin{equation}\label{b}
\sum_{h,x}\mu_{H}(h)\Delta_{ah_1,bh_2}f(x).
\end{equation}
One can think of this as an arithmetic analogue of the two-dimensional `box norm' \eqref{real box norm}.  In our eventual application we are able to ensure that $a$ and $b$ are a generic pair of integers from the interval $[N^{1/2}]$.  In particular, at least one of them has size proportional to $N^{1/2}$ and their highest common factor is small.  One may think of this as a proxy for linear independence. 

We begin by characterising largeness of \eqref{b} when the directions are coprime.

\begin{lemma}[Inverse theorem for the arithmetic box norm]
Let $a,b$ be positive integers with $\gcd(a,b) = 1$. 
Suppose that $f:\Z\to\C$ is $1$-bounded with support in the interval $[N]$ and satisfies
\begin{equation}\label{arith1}
\sum_{h,x}\mu_{H}(h)\Delta_{ah_1,bh_2}f(x)\geq \delta N.
\end{equation}
Then there exist 1-bounded functions $g, h:\Z\to\C$ such that
\begin{itemize}
\item $g$ is $a$-periodic, in the sense that $g(x+a) = g(x)$ for all $x$;
\item $h$ is approximately $b$-periodic, in the sense that for any $\eps > 0$ we have
$$
\hash\set{x \in [N] : h(x+by) \neq h(x) \text{ for some } |y| \leq \eps N/b}  \leq \brac{1+\tfrac{2\eps N}{b}}\brac{1 + \tfrac{N}{a}};
$$
\end{itemize}
and furthermore
\begin{equation}\label{1st box inverse}
\biggabs{\sum_{x}f(x)g(x)h(x)}\geq \delta \floor{H}^2  - 2\brac{\tfrac{H}{a} + \tfrac{Hb}{N}}\floor{H}^2.
\end{equation}
\end{lemma}
\begin{remark}
In parsing the above inequalities, it may be helpful to keep in mind that in our application $a$, $b$ and $H$ are of order $\sqrt{N}$, with $H$ considerably smaller than $a$, in which case the lower bound in \eqref{1st box inverse} becomes $\Omega(\delta H^2)$.
\end{remark}
\begin{proof}
The majority of our proof is concerned with manipulating \eqref{arith1} until we can interpret it as a genuine box norm \eqref{real box norm}, and thereby apply the box norm inverse theorem. The essential observation is that, since $\gcd(a,b)=1$, every integer $x$ can be uniquely represented in the form
\[
x=ay+bz\qquad (y\in\Z,\ z\in[a]).
\]
We note that if $x \in [N]$ then the constraint on $z$ forces $y$ to lie in the range $-b < y < N/a$.

Defining $F:\Z\times\Z\to\C$ by $F(y,z):=f(ay+bz)$, the left-hand side of \eqref{arith1} becomes
$$
\sum_{y,y' \in \Z} \sum_{\substack{z \in [a]\\ z' \in \Z}}F(y, z)\overline{F(y', z)}
\overline{F(y, z')}F(y', z')\mu_H(y'-y)\mu_H(z'-z).
$$
If $z'$ and $z$ contribute to the above sum then
$
z' \in z + (-H, H) \subset (-H+1, a +H).
$
Hence we can restrict the range of summation of $z'$ to $[a]$, at the cost of perturbing the sum by at most
$
2\floor{H}(\frac{N}{a} + b).
$  
It follows that
\begin{multline*}
\biggabs{\sum_{y,y'}\sum_{z, z'\in [a]} F(y, z)\overline{F(y', z)}
\overline{F(y, z')}F(y', z')\mu_H(y'-y)\mu_H(z'-z)}\\ \geq \delta N - 2\floor{H}\brac{\tfrac{N}{a} + b}.
\end{multline*}

We remove the Fej\'er kernels by Fourier expansion:
\begin{multline*}
\sum_{\substack{y,y' \\ z,z'\in[a]}}F(y,z)\overline{F(y',z)F(y,z')}F(y',z')\mu_H(y'-y)\mu_H(z'-z) =\\
\int_{\T^2}\sum_{\substack{y,y' \\ z,z'\in[a]}}F(y,z)\overline{F(y',z)F(y,z')}F(y',z')\hat{\mu}_H(\alpha)\hat{\mu}_H(\beta)e(\alpha(y'-y)+\beta(z'-z))\intd\alpha\intd\beta \\
\leq\left(\int_{\T}|\hat{\mu}_H(\alpha)|\intd\alpha\right)^2\sup_{\alpha, \beta\in\T}\biggabs{\sum_{\substack{y,y' \\ z,z'\in[a]}}F(y,z)F_2(y',z)F_3(y,z')F_4(y',z')},
\end{multline*}
where $F_2(y',z):=\overline{F(y',z)}e(-\beta z)$, $F_3(y,z'):=\overline{F(y,z')}e(-\alpha y)$, and $F_4(y',z')$ $:=F(y',z')e(\alpha y'+\beta z')$.

We observe that $\hat{\mu}_H(\alpha)=|\hat{1}_{[H]}(\alpha)|^2/\floor{H}^2$, which implies that $\int_\T|\hat{\mu}(\alpha)|d\alpha=\floor{H}^{-1}$. Therefore
\begin{equation}\label{arith2}
\biggabs{\sum_{\substack{y,y' \\ z,z'\in[a]}}F(y,z)F_2(y',z)F_3(y,z')F_4(y',z')} \geq \delta \floor{H}^2 N - 2\floor{H}^3\brac{\tfrac{N}{a} + b},
\end{equation}
for $1$-bounded functions $F_i:\Z\times[a]\to\C$ of the form $F_i(y,z)=f(ay+bz)e(\alpha_1y+\alpha_2z)$. Since $f$ is supported on $[N]$, there are exactly $N$ pairs $(y',z')\in\Z\times[a]$ for which $F(y',z')\neq 0$. Thus, by pigeonholing in $y'$ and $z'$ in~(\ref{arith2}) and setting $L(y):=F_3(y,z')$ and $R(z):=F_2(y',z)F_4(y',z')$, we get that
\[
\biggabs{\sum_{y}\sum_{z\in[a]}F(y,z)L(y)R(z)}\geq \delta \floor{H}^2 - 2\floor{H}^3\brac{\tfrac{1}{a} + \tfrac{b}{N}}.
\]

For each $x\in\Z$, define $l(x)\in\Z$ and $r(x)\in[a]$ by $x=al(x)+br(x)$, and set $g(x):=R\circ r(x)$ and $h(x):=L\circ l(x)$. Then
it remains to check the invariance properties of $g$ and $h$. To see that $g(x)=g(x+ay)$ for all $x,y\in\Z$, just note that $r(x)=r(x+ay)$ for every $x,y\in\Z$. 

Finally we establish that, for most $x\in[N]$, we have $h(x)=h(x+bz)$ for all $|z|\leq \ve N/b$.  First note that $l(x)=l(x+bz)$ whenever $\ve N/b < r(x)\leq a-\ve N/b$. Hence for this to fail, $x$ must lie in one of at most $1+2\ve N/b$ congruence classes modulo $a$. The number of such $x$ lying in the interval $[N]$ is at most
\[
\brac{1+\frac{2\eps N}{b}}\brac{1 + \frac{N}{a}}.% \ll \frac{\eps N^2}{ab} + \frac{N}{a}.
\]
%The bound claimed in the lemma follows form our size assumptions on $ab$ and $a$. 

\end{proof}

The lemma also yields a result in the situation in which $\gcd(a,b)>1$.  In proving this we take the opportunity to smooth out the $b$-invariance of $h$ slightly, whilst also giving an explicit description of $h$ in terms of $f$.  More concretely, we replace $h$ with a projection of $fg$ onto cosets of $b\cdot \Z$.
%\begin{definition}[$\eta$-Lipschitz along $b\cdot \Z$]
%We say that a 1-bounded function $h : \Z \to \C$ is \emph{$\eta$-Lipschitz along $b \cdot \Z$ }if for any $x,y \in \Z$ we have
%$$
%|h(x+by) - h(x)|\leq \eta  |y|.
%$$
%\end{definition}
\begin{lemma}\label{arithcor} 
There exists an absolute constant $c>0$ such that on assuming  $1 \leq H \leq c\delta^3 N^{1/2}$ and $1 \leq K \leq c\delta^2 H^2 N^{-1/2}$ the following holds. Let $a,b\in [N^{1/2}]$ with $\gcd(a,b)\leq \delta^{-1}$ and $a,b \geq \delta N^{1/2}$.  Suppose that $f:\Z\to\C$ is $1$-bounded, supported on the interval $[N]$, and satisfies
\begin{equation*}%\label{arith1}
\biggabs{\sum_{h, x}\mu_{H}(h)\Delta_{ah_1,bh_2}f(x)}\geq\delta N.
\end{equation*}
Then there exists a 1-bounded $a$-periodic function $g$ such that 
\begin{equation}\label{improved correlation}
\sum_{x}  f(x) g(x)\sum_k\mu_K(k) \overline{f(x+bk) g(x+bk)} \gg \delta^2 H^4/N.
\end{equation}
\end{lemma}

\begin{proof}
Set $q := \gcd(a,b) \leq \delta^{-1}$.  For each $u\in [q]$, define a $1$-bounded function $f_u:\Z\to\C$ by $f_u(x):=f(u +qx)$, and let $I_u := \set{x : u + qx \in [N]}$ denote the interval on which $f_u$ is supported. By the pigeon-hole principle, for some $u$ we have
\[
\sum_{x,h_1, h_2}\mu_{H}(h_1)\mu_{H}(h_2)\Delta_{\f{a}{q}h_1,\f{b}{q} h_2}f_u(x)\geq\delta |I_u|.
\]

Note that $\gcd(a/q,b/q)=1$, so by the previous lemma, there exist 1-bounded functions $g_u,h_u:\Z\to\C$ such that
\[
\biggabs{\sum_xf_u(x)g_u(x)h_u(x)}\geq \delta \floor{H}^2 - 2\brac{\tfrac{Hq}{a} + \tfrac{Hb}{q|I_u|}}\floor{H}^2 \gg \delta H^2.
\]
Furthermore, $g_u$ is $(a/q)$-periodic and
\begin{multline*}
\#\set{x\in I_u:h_u(x)\neq h_u(x+yb/q)\text{ for some }|y| \leq \eps |I_u|q/b}\\ \leq \brac{1+\tfrac{2q\eps |I_u|}{b}}\brac{1 + \tfrac{q|I_u|}{a}} \ll \tfrac{N}{a} + \tfrac{\eps N^2}{ab}.%\delta^{-1}N(\eps + N^{-1/2}).
\end{multline*}

Defining $g_{u'}$ and $h_{u'}$ to be identically zero when $u'\neq u$, we set $g(u'+qx):=g_{u'}(x)$ and $h(u'+qx):=h_{u'}(x)$.  One can then check that $g$ is $a$-invariant, that
$$
\biggabs{\sum_xf(x)g(x)h(x)} \gg \delta H^2,
$$
and that
$$
\#\set{x\in[N]:h(x)\neq h(x+by)\text{ for some }|y| \leq \eps N/b}\ll \tfrac{N}{a} + \tfrac{\eps N^2}{ab}.
$$

We may use the latter property to show that, provided $K \geq 1$, we have
$$
\biggabs{\sum_{x} f(x) g(x) h(x) - \sum_{x} h(x) \E_{y \in [K]} g(x+by) f(x + by) } \ll \tfrac{NK}{a} .
$$
Provided that  $K \leq c\delta^2 H^2 N^{-1/2}$ we deduce that
$$
\biggabs{\sum_xh(x) \E_{y \in [K]} g(x+bk) f(x + bk)} \gg \delta H^2.
$$
One can check that, as a function of $x$, the inner expectation is 1-bounded with support in $[-2N, 2N]$.  Applying the Cauchy--Schwarz inequality and changing variables then gives \eqref{improved correlation}.
\end{proof}
Finally we observe that a function of the form 
\begin{equation}\label{h def}
h(x) := \sum_k \mu_K(k) f(x+by)
\end{equation} 
has nice $b$-periodicity properties.

\begin{lemma}\label{h lipschitz}
If $h$ is defined as in \eqref{h def} for some 1-bounded $f$, then $h$ is $O(K^{-1})$-Lipschitz along $b \cdot \Z$, in that for any $x, y\in \Z$ we have $h(x+by) = h(x) + O(|y|/K)$.
\end{lemma}

\begin{proof} 
Recalling the definition \eqref{fejer}, note that $\mu_K$ is $(2/\floor{K})$-Lipschitz, in that $|\mu_K(k+y) - \mu_K(k)| \leq 2|y|/\floor{K}$ for all $k, y \in \Z$.  Hence, for $|y| \leq K$, a change of variables gives
$$
|h(x+by) - h(x)|  \leq  \sum_k |\mu_K(k-y) - \mu_K(k)| \ll \frac{|y|}{K}\sum_{|k| < 2K} 1.
$$
\end{proof}

\section{Quantitative concatenation}\label{concat sec}

The endpoint of this section is to show how our counting operator \eqref{counting op} is controlled by the $U^5$-norm.  We begin with four technical lemmas.  The first says that
convolving Fej\'er kernels along progressions of coprime common difference covers a substantial portion of an interval in a somewhat regular manner, a fact that can be interpreted Fourier analytically in the following.

\begin{lemma}\label{L1 Fourier bound}
Let $K, L \geq 1$ and let $a, b$ be integers satisfying $a \geq \delta L$, $b \geq \delta K$ and $\gcd(a, b) \leq \delta^{-1}$.  Then
$$
\int_\T \bigabs{\widehat{\mu}_K(a\beta)}\bigabs{\widehat{\mu}_L(b\beta)}\intd\beta \ll \frac{\delta^{-4}}{\floor{K}\floor{L}} .
$$
\end{lemma}

\begin{proof}
Expanding Fourier transforms, one can check that
\begin{multline*}
\int_\T \bigabs{\widehat{\mu}_H(a\beta)}\bigabs{\widehat{\mu}_K(b\beta)}\intd\beta\\
 = \floor{K}^{-2} \floor{L}^{-2}\hash\biggset{(x, y) \in [K]^{2}\times[L]^{2} : a(x_1 - x_2) = b(y_1 - y_2)}.
\end{multline*}
Writing $d := \gcd(a,b)$, the number of solutions to the equation is at most
$$
\floor{K}\floor{ L} \brac{\tfrac{\floor{K}}{b/d} + 1}\brac{\tfrac{\floor{L}}{a/d} + 1}.
$$
\end{proof}

Our next lemma allows us to discard pairs of integers $a,b$ which are not sufficiently coprime. We exploit this repeatedly.

\begin{lemma}\label{gcd}
For fixed integers $0 \leq a_1, a_2 \leq M$.  The number of pairs $(b,c)$ of integers $0 \leq b,c \leq M$ such that $\gcd(a_1+b,a_2+c)> \delta^{-1}$ is $\ll \delta M^2$.
\end{lemma}
\begin{proof}
Notice that if $d = \gcd(a_1+b, a_2+c)$ then $d \leq 2M$.  Hence
\begin{align*}
\sum_{\substack{0 \leq b,c \leq M\\ \gcd(a_1+b,a_2+c) > \delta^{-1}}} 1  \leq \sum_{\delta^{-1} < d \leq 2M}\ \biggbrac{\ \sum_{0\leq m \leq 2M,\ d \mid m} 1}^2
 &\leq  \sum_{\delta^{-1} < d \leq 2M} \brac{\frac{2M}{d} + 1}^2\\ & \ll M^2 \sum_{d> \delta^{-1}} \recip{d^2}
  \ll \delta M^2 .
\end{align*}
\end{proof}

The following lemma says that, as $a$ and $h$ range over $[N^{1/2}]$, the difference function $\Delta_{ah} f$ behaves like $\Delta_k f$ with $k \in [N]$, at least on average.

\begin{lemma}\label{densifying difference functions}
Let $f : \Z \to \C$ be a 1-bounded function with support in $[N]$.  Suppose that $\delta N^{1/2} \leq H \leq N^{1/2}$ and 
$$
 \E_{a\in[ N^{1/2}]}\sum_h \mu_H(h) \norm{\Delta_{ah} f }_{U^s}^{2^s} \geq \delta  \norm{1_{[N]}}_{U^s}^{2^s}.
$$ 
Then
$$
\norm{f}_{U^{s+1}}^{2^{s+1}} \gg \delta^{12}  \norm{1_{[N]}}_{U^{s+1}}^{2^{s+1}}
$$
\end{lemma}

\begin{proof}  Expanding the definition of the $U^s$-norm 
\begin{multline*}
\E_{a\in [N^{1/2}]}\sum_h \mu_H(h) \norm{\Delta_{ah} f }_{U^s}^{2^s}\\ = \sum_{h_1, \dots, h_s, x} \overline{\Delta_{h_1, \dots, h_s}f(x)} \E_{a\in [N^{1/2}]}\sum_h \mu_H(h) \Delta_{h_1, \dots, h_s} f(x+ ah).
\end{multline*}
Employing the Cauchy--Schwarz inequality to double the $a$ and $h$ variables gives
$$
\E_{a, a' \in [N^{1/2}]}\sum_{h_i} \sum_x \sum_{h,h'} \mu_H(h)\mu_H(h')  \Delta_{h_1, \dots , h_s, ah - a'h'}f(x) \gg \delta^2  N^{s+1}. 
$$

By Lemma \ref{gcd} and the pigeon-hole principle, we deduce the existence of $a, a' \gg \delta^2N^{1/2}$ with $\gcd(a, a') \ll \delta^{-2}$  such that
$$
\sum_{h_i} \sum_x \sum_{h,h'} \mu_H(h)\mu_H(h')  \Delta_{h_1, \dots , h_s, ah - a'h'}f(x) \gg \delta^2 N^{s+1} .
$$
By Fourier inversion and extraction of a large Fourier coefficient, there exists $\alpha \in \T$ such that the right-hand side above is at most
$$
  \int_\T \abs{\widehat{\mu}_H(a\beta)}\abs{\widehat{\mu}_H(a'\beta)}\intd\beta\biggabs{\sum_{h_i} \sum_x  \Delta_{h_1, \dots , h_s, h_{s+1}}f(x)e(\alpha h_{s+1})}.
$$
The result follows on employing Lemma \ref{L1 Fourier bound} and Lemma \ref{phase invariance}.
\end{proof}

We now prove a similar lemma, but with $\Delta_{ah} f$  replaced by  $fg_a$ where $g_a$ is $a$-periodic.  The moral is that these are similar quantities (on average).
\begin{lemma}\label{difference vs periodic product}
Let $f, g_a : \Z \to \C$ be 1-bounded functions such that $g_a$ is $a$-periodic and $\supp(f) \subset[N]$.  Suppose that 
$$
\E_{a\in[ N^{1/2}]} \norm{f g_a}_{U^s}^{2^s} \geq \delta  \norm{1_{[N]}}_{U^s}^{2^s}.
$$ 
Then
$$
\norm{f}_{U^{s+1}}^{2^{s+1}} \gg \delta^{24}  \norm{1_{[N]}}_{U^{s+1}}^{2^{s+1}}
$$
\end{lemma}

\begin{proof}
Fix $a \in [N^{1/2}]$. By the periodicity of $g_a$ and a change of variables, we have
$$
\sum_{h_i} \sum_x \Delta_{h_1, \dots, h_s}g_a(x)\Delta_{h_1,\dots, h_s}f(x) = \sum_{h_i} \sum_x \Delta_{h_1, \dots, h_s}g_a(x)\E_{y \in [N^{1/2}]}\Delta_{h_1, \dots, h_s}f(x + ay).
$$
Notice that the sum over $x$ is non-zero only if $|x|, |h_i| < N$, hence by Cauchy--Schwarz and a change of variables
\begin{align*}
\biggbrac{\E_{a\in [N^{1/2}]}\norm{fg_a}_{U^s}^{2^s}}^2 & \ll N^{s+1}  \E_{a\in [N^{1/2}]}\sum_{h_i} \sum_x \sum_y \mu_{N^{1/2}}(y) \Delta_{h_1, \dots, h_s, ay}f(x)\\
& =N^{s+1} \E_{a\in [N^{1/2}]}\sum_y \mu_{N^{1/2}}(y) \norm{\Delta_{ay}f}_{U^{s}}^{2^s}
\end{align*}
The result follows on employing Lemma \ref{densifying difference functions}.
\end{proof}

We are now ready to give the technical heart of this section.  The (somewhat lengthy) assumptions come from our eventual application of Lemma \ref{arithcor}.
%In reading the following  it may be useful to keep in mind that $a$ is fixed, whilst we sum over $b \in [a, \sqrt{N}]$.
\begin{lemma}\label{hb lemma}
Fix $a \in \N$ and let $\delta N^{1/2} \leq K \leq N^{1/2}$.  For each $b \in [N^{1/2}]$ let $f, g_b, h_b : \Z \to \C$ be 1-bounded functions such that $\supp(f), \supp(h_b) \subset [N]$ and where $g_b$ is $b$-periodic.  Set
$$
\tilde{h}_b(x) := \sum_k \mu_K(k) h_{b}(x+(a+b)k)
$$
and suppose that
$$
\sum_{\substack{\delta \sqrt{N}\leq b \leq\sqrt{N}\\ \gcd(a,b) \leq \delta^{-1}}}  \sum_x f(x)g_b(x)\tilde{h}_b(x) \geq \delta N^{3/2}.
$$
Then 
$$
\E_{ b \in [ N^{1/2}]} \bignorm{h_b}_{U^3}^8 \gg \delta^{208}  \norm{1_{[N]}}_{U^3}^8.
$$
\end{lemma}

\begin{proof}
To ease notation, write
$$
\tilde{h}_b(x) := \sum_k \mu_K(k) h_{b}(x+(a+b)k)
$$
We apply Cauchy--Schwarz to remove the weight $f(x)$ and double the $b$ variable, yielding
$$
\sum_{\substack{\delta \sqrt{N}\leq b,b' \leq\sqrt{N}\\ \gcd(a,b) \leq \delta^{-1} }} \sum_x g_b(x) \tilde{h}_{b}(x)\overline{g_{b'}(x) \tilde{h}_{b'}(x)} \geq \delta^2  N^2.
$$
Employing Lemma \ref{gcd}, we may discard those $b,{b'}$ for which one of $\gcd(b', a+{b})$ or $\gcd(a+b', a+{b})$ is greater than $C\delta^{-2}$.   On combining this with the popularity principle, we deduce the existence of $\mathcal{B} \subset [\delta N^{1/2}, N^{1/2}]$ of size $|\mathcal{B}| \gg \delta^2 N^{1/2}$ such that for each $b \in \mathcal{B}$ there exists $b' \in [N^{1/2}]$ with all of $\gcd(b, a+{b})$, $\gcd({b'}, a+{b})$, $\gcd(a+b', a+{b})$ at most $O(\delta^{-2})$ and satisfying
\begin{equation}\label{PHbb'}
\sum_x g_b(x) \overline{\tilde{h}_{b'}(x)g_{b'}(x)} \tilde{h}_{b}(x) \gg \delta^2 N.
\end{equation}

Expanding the definition of $\tilde{h}_{b'}$, using the invariance of $g_b$ and changing variables gives
\begin{multline*}
 \sum_x \E_{k_1, k_3 \in [K]}\sum_{k_2}\mu_K(k_2) g_b(x+ (a+b')k_2 + {b'}k_3) \overline{h_{b'}(x+bk_1 + {b'}k_3)}\\ \overline{g_{b'}(x+bk_1 + (a+b')k_2)}\ \tilde{h}_{b}(x+bk_1 + (a+b')k_2 + {b'}k_3) \gg \delta^2  N .
\end{multline*}
Since $h_{b'}$ is supported on $[N]$ and $b, {b'}, K \leq  N^{1/2}$, there are at most $O(N)$ values of $x$  which contribute to the above sum.  Applying H\"older's inequality then gives
\begin{multline*}
 \sum_x \biggbrac{\E_{k_1, k_3 \in [K]}\sum_{k_2}\mu_K(k_2) g_b(x+ (a+b')k_2 + {b'}k_3) \overline{h_{b'}(x+bk_1 + {b'}k_3)}\\ \overline{g_{b'}(x+bk_1 + (a+b')k_2)}\ \tilde{h}_{b}(x+bk_1 + (a+b')k_2 + {b'}k_3) }^8\gg \delta^{16}  N .
\end{multline*}
The sum inside the 8th power corresponds to an  integral with respect to three probability measures on $\Z$, with integrand amenable to Lemma \ref{box cauchy}.  Combining this with a change of variables gives
$$
 \sum_x\sum_{k_1, k_2, k_3}\mu_K(k_1)\nu_K(k_2)\mu_K(k_3) \Delta_{bk_1, (a+b')k_2, {b'}k_3}\ \tilde{h}_{b}(x) \gg \delta^{16}  N ,
$$
where we set
$$
\nu_K(k) : = \sum_{k_1 - k_2 = k} \mu_K(k_1)\mu_K(k_2).
$$

By Lemma \ref{h lipschitz}, each $\tilde{h}_b$ is $O(K^{-1})$-Lipschitz along $(a+b)\cdot \Z$.  Hence, if $l_i \in [L]$, a telescoping identity shows that
$$
|\Delta_{h_1+(a+{b})l_1, h_2+(a+{b})l_2, h_3+(a+{b})l_3} \tilde{h}_{b}(x)- \Delta_{h_1, h_2, h_3} \tilde{h}_{b}(x)| \ll L/K.
$$
Taking $L := c \delta^{16} K$ we obtain 
\begin{multline*}
 \sum_x\sum_{k_1, k_2, k_3}\mu_K(k_1) \nu_K(k_2) \mu_K(k_3)\E_{l_1, l_2, l_3\in [L]}  \\ \Delta_{bk_1 + (a+{b})l_1,\, (a+b')k_2+ (a+{b})l_2,\, {b'}k_3+ (a+{b})l_3}\ \tilde{h}_{b}(x) \gg \delta^{16}  N .
\end{multline*}
We may replace the uniform measure on the $l_i$ by Fej\'er kernels at the cost of three applications of Cauchy--Schwarz; this gives
\begin{multline*}
 \sum_x\sum_{\substack{k_1, k_2, k_3\\l_1, l_2, l_3}} \mu_K(k_1) \nu_K(k_2) \mu_K(k_3) \mu_L(l_1) \mu_L(l_2) \mu_L(l_3)\\ \Delta_{bk_1 + (a+{b})l_1,\, (a+b')k_2+ (a+{b})l_2,\, {b'}k_3+ (a+{b})l_3}\ \tilde{h}_{b}(x) \gg \delta^{128}  N .
\end{multline*}

Write 
\begin{align*}
\lambda_1(h)   := \sum_{bk + (a+{b})l = h}& \mu_K(k) \mu_L(l),\qquad
\lambda_2(h)  := \sum_{(a+b')k + (a+{b})l = h} \nu_K(k) \mu_L(l),\\
& \lambda_3(h)  := \sum_{{b'}k + (a+{b})l = h} \mu_K(k) \mu_L(l).
\end{align*}
Then
$$
 \sum_x\sum_{h_1, h_2, h_3} \lambda_1(h_1)\lambda_2(h_2)\lambda_3(h_3) \\ \Delta_{h_1, h_2, h_3}\ \tilde{h}_{b}(x) \gg \delta^{128}  N .
$$

By Fourier inversion and extraction of a large Fourier coefficient, there exist $\alpha_i \in \T$ such that 
$$
\biggabs{ \sum_x\sum_{h_1, h_2, h_3}  \Delta_{h_1, h_2, h_3}\ \tilde{h}_{b}(x)e(\underline{\alpha} \cdot \underline{h})}\prod_{i=1}^3 \int_\T \bigabs{\widehat{\lambda}_i(\beta)}\intd\beta \gg \delta^{128}  N .
$$
By our choice of $b$, $b'$ (see the paragraph preceding \eqref{PHbb'}), together with Lemma \ref{L1 Fourier bound}, for each $i$ we have 
\begin{equation}\label{fejer fourier bound}
\int_\T \bigabs{\widehat{\lambda}_i(\alpha)}\intd\alpha \ll \frac{\delta^{-8}}{KL} \ll \frac{\delta^{-26}}{ N},
\end{equation}
the latter following from the fact that $L \gg c\delta^{16} K$ and $K \geq \delta N^{1/2}$.
On combining this with Lemma \ref{phase invariance} we obtain
$$
\bignorm{\tilde{h}_{b}}_{U^3}^8 \gg \delta^{206} N^4.
$$
Since $\tilde{h}_b$ is an average of translates of $h_b$, we may apply the triangle inequality for the $U^3$-norm, together with the fact that Gowers norms are translation invariant, and conclude that $\norm{h_b}_{U^3}^8 \gg \delta^{206} N^4$. Summing over $b \in \mathcal{B}$ gives our final bound.
\end{proof}

Finally we synthesise Lemmas \ref{linearisation}, \ref{arithcor} and \ref{hb lemma}.

\begin{theorem}[Global $U^5$-control]\label{global U5}
Let $g_0, g_1, f : \Z \to \C$ be 1-bounded functions, each with support in $[N]$.  Suppose that %$N \geq Cq\delta^{-C}$ and that
$$
\abs{\Lambda_{q, N}(g_0, g_1, f)} \geq \delta \Lambda_{q, N}(1_{[N]}).
$$
Then 
$$
\sum_{u \in [q]}\norm{f}_{U^5(u + q  \Z)}^{2^5} \gg \delta^{2^{25}} \sum_{u \in [q]}\norm{1_{[N]}}_{U^5(u+ q  \Z)}^{2^5}.
$$
\end{theorem}

\begin{proof} We recall our convention \eqref{M def} regarding $M$.
We begin by applying the linearisation procedure (Lemma \ref{linearisation}) to deduce that
$$
 \sum_{ a,b\in (-2M,2M)}\ \biggabs{\sum_{h}\mu_{H}(h)\sum_x\Delta_{q(a+b)h_1,qbh_2,qah_3}f(x)}\\\gg \delta^{32} NM^2.
$$
We note that the sum inside the absolute value is invariant under $a \mapsto -a$. Hence we may restrict to $a, b \in [0, 2M]$ at the cost of changing the absolute constant. Applying Lemma \ref{gcd} we may discard those $a,b$ for which either $\gcd(a,b) > C\delta^{-32}$ or $b < c\delta^{32} M$.
Partitioning the sum over $x$ into congruence classes $u \bmod q$, the popularity principle gives:
\begin{itemize}
\item at least $\Omega(\delta^{32} q)$ residues $u \in [q]$;
\item for each of which there is a subset of $h_3 \in (-H, H)$ of $\mu_H$-measure\footnote{i.e.\ $\sum_{h_3 \in \mathcal{H}} \mu_H(h_3) \gg \delta^{32}$.} at least $\Omega(\delta^{32})$; 
\item for each of which there exist $\Omega(\delta^{32} M)$ values of $a \in [2M]$;
\item for each of which there are $\Omega(\delta^{32}M)$ values of $b \in [2M]$ satisfying $\gcd(a,b) \ll \delta^{-32}$ and $b \gg \delta^{32} M$;
\end{itemize}
and together these satisfy
\begin{equation*}%\label{linearised1}
 \biggabs{\sum_{h_1, h_2}\mu_{H}(h_1, h_2)\sum_x\Delta_{(a+b)h_1,bh_2,ah_3}f(qx-u)}\\\gg \delta^{32} M^2.
\end{equation*}

For fixed $u, h_3, a$ write 
$
\tilde{f}(x) := \Delta_{ah_3} f(qx-u),
$ so that $\tilde{f}$ has support in the interval $[(2M)^2]$ and
\begin{equation*}%\label{linearised1}
 \biggabs{\sum_{h_1, h_2}\mu_{H}(h_1, h_2)\sum_x\Delta_{(a+b)h_1,bh_2}\tilde{f}(x)}\\\gg \delta^{32} M^2.
\end{equation*}

Set
\begin{equation}\label{1st H bound}
H:= c \delta^{96} M \qquad \text{and}\qquad K:= c^3 \delta^{160} M,
\end{equation}
with $c$ sufficiently small to ensure that we may apply Lemma \ref{arithcor}. This gives the existence of a 1-bounded $b$-periodic function $g_b$ such that on setting
\begin{equation}\label{hb defn}
\tilde{h}_b(x) := \sum_k\mu_K(k)  \overline{\tilde{f}(x+(a+b)k)g_b(x+(a+b)k)} 
\end{equation}
%with support in $[N]$ and 
we have
\[
\sum_x\tilde{f}(x)g_b(x)\tilde{h}_b(x)\gg \delta^{448} M^2.
\]

Setting $\eta := c \delta^{480}$ for some small absolute constant $c >0$, we may sum over our set of permissible $b$ to deduce that
$$
\sum_{\substack{\eta M \leq b \leq 2M\\ \gcd(a, b) \leq \eta^{-1}}} \sum_x\tilde{f}(x)g_b(x)h_b(x)\geq \eta M^3.
$$
The hypotheses of Lemma \ref{hb lemma} having been met, we conclude that
$$
\E_{ b \in [2M]} \bignorm{\tilde{f}g_b}_{U^3}^8 \gg \delta^{99,840}  \norm{1_{[M^2]}}_{U^3}^8.
$$
Applying Lemma \ref{difference vs periodic product} then gives
$$
\bignorm{\tilde{f}}_{U^4}^{16} \gg \delta^{2,396,160}  \norm{1_{[M^2]}}_{U^4}^{16}.
$$

Recalling that $\tilde{f}(x) = \Delta_{ah_3} f_u(x)$ where $f_u(x) := f(qx-u)$, we may integrate over the set of permissible $h_3 $ and $a$, utilising positivity to extend the range of summation, and deduce that
$$
\E_{a \in [2M]} \sum_h \mu_H(h_3) \bignorm{\Delta_{ah_3} f_u}_{U^4}^{16} \gg \delta^{2,396,224}  \norm{1_{[M^2]}}_{U^4}^{16}
$$
Using Lemma \ref{densifying difference functions} and summing over the permissible range of $u$ we get that
$$
 \E_{u\in [q]}\norm{f_u}_{U^5}^{32} \gg \delta^{28,754,720}  \norm{1_{[M^2]}}_{U^5}^{32},
$$
and the result follows.
\end{proof}

\section{Degree lowering}\label{sec6}
So far, we have shown that $\Lambda_{q, N}(f_0,f_1,f_2)$ is controlled by $\E_{u\in[q]}\|f_2\|_{U^5(u+q\Z)}^{2^5}$ whenever $f_0,f_1,$ and $f_2$ are $1$-bounded complex-valued functions supported on the interval $[N]$. The next step in our argument is to bound $\Lambda_{q, N}(f_0,f_1,f_2)$ in terms of the $U^5(u+q\Z)$-norm of the dual function
\begin{equation}\label{eq6.1}
F(x):=\E_{y\in[M]}f_0(x-qy^2)f_1(x+y-qy^2).
\end{equation}
We postpone this deduction until \S\ref{inverse theorem section}. In this section we show how $U^5$-control of the dual implies $U^2$-control.  

Our argument combines three simple lemmas: Weyl's inequality; what we call `dual--difference interchange', which allows us to replace the difference function of the dual by the dual of the difference functions; and the fact that a function whose difference functions correlate with `low rank' Fourier coefficients must have a large uniformity norm of lower degree.

The following log-free variant of Weyl's inequality can be found in \cite[Lemma A.11]{GreenTaoQuadratic}.

\begin{lemma}[Weyl's inequality]\label{weyl-ineq}
There exists an absolute constant $C$ such that the following holds. Let $\alpha,\beta \in \T$, $\delta \in (0,1)$ and let $I \subset \Z$ be an interval with
$|I| \geq C\delta^{-6}$ and
\[ \big| \E_{y \in I} e(\alpha y^2 + \beta y ) \big| \geq \delta.\]
Then there exists a positive integer $q \ll \delta^{-4}$ such that  
$$\|q\alpha\| \ll \delta^{-14}|I|^{-2}.$$
\end{lemma}
This has the following consequence, which uses our convention \eqref{M def} regarding $M$.
\begin{lemma}\label{lem6.2}
There exist an absolute constant $C$ such that for $N \geq C(q/\delta)^C$ the following holds.  Suppose that for $\alpha \in \T$ there are $1$-bounded functions $g_0,g_1:\Z\to\C$ supported on the interval $[N]$ such that
\[
\left|\sum_x\sum_{y\in[M]}g_0(qx)g_1(qx+y)e(\alpha (x+y^2))\right|\geq\delta M N/q.
\]
Then there exists a positive integer $q'\ll\delta^{-4}$ such that $\|q'q^2\alpha\|\ll\delta^{-14}q^3/N$.
\end{lemma}
\begin{proof}
We split the sum over $y\in[M]$ into arithmetic progressions modulo $q$ and split the sum over $x $ into intervals of length $M/q$.  Hence, by the pigeon-hole principle, there exists $u \in [q]$ and an integer $m$ such that on rounding the sum over $y$ we have 
$$
\left|\sum_{x, y\in[M/q]}g_0(q(m+x))g_1(u+q(m+x+y))e\brac{\alpha \brac{x+(u+qy)^2}}\right| \\ \gg\delta(M/q)^2.% - 2\brac{\frac{N}{q} + 1}.
$$
Define the functions
\begin{align*}
h_0(x) := g_0(q(m+x)) & e(\alpha x) 1_{[M/q]}(x),  \qquad h_1(x):= g_1(u+q(m+x))1_{[2M/q]},\\  &h_2(x) := e\brac{\alpha (u+qx)^2}1_{[M/q]}(x)
\end{align*}
Then by orthogonality, extraction of a large Fourier coefficient and Parseval we have
\begin{align*}
\delta M^2/q^2 \ll \left|\int_\T \hat{h}_0(\beta) \hat{h}_1(-\beta) \hat{h}_2(\beta) \intd\alpha\right| \ll \bignorm{\hat{h}_2}_\infty \bignorm{\hat{h}_0}_{L^2} \bignorm{\hat{h}_1}_{L^2} \ll \bignorm{\hat{h}_2}_\infty M/q.
\end{align*}

It follows that there exists $\beta \in \T$ such that
$$
\abs{\sum_{x \in [M/q]} e\brac{\alpha (u+qx)^2+ \beta x}} \gg\delta M/q.
$$
Applying Weyl's inequality, we deduce the existence of $q' \ll \delta^{-4}$ such that $\norm{q' q^2 \alpha} \ll \delta^{-14} (q/M)^2$.
\end{proof}

\begin{lemma}[Dual--difference interchange]\label{dual difference interchange}
For each $y \in [M]$, let $F_y : \Z \to \C$ be a 1-bounded function with support in an interval of length $N$.  Set
$$
F(x) := \E_{y \in [M]} F_y(x).
$$
Then for any function $\phi: \Z^s \to \T$ and  finite set $\mathcal{H}\subset \Z^s$ we have
\begin{multline*}
\brac{N^{-s-1}\sum_{\vh \in \mathcal{H}}\abs{ \sum_x  \Delta_{\vh} F(x) e\bigbrac{\phi(\vh)x }}}^{2^s} \ll_s \\N^{-2s-1} \sum_{\vh^{0}, \vh^{1}\in \mathcal{H}}\abs{\sum_x \E_{y \in [M]} \Delta_{\vh^0-\vh^1} F_y(x) e\bigbrac{\phi(\vh^0; \vh^1)x }},
\end{multline*}
where
$$
\phi(\vh^0; \vh^1) := \sum_{\omega \in \set{0, 1}^s} (-1)^{|\omega|} \phi(\vh^{\omega})\qquad \text{and} \qquad
\vh^{\omega} := (h_1^{\omega_1}, \dots, h_s^{\omega_s}).
$$
\end{lemma}

\begin{proof}
We proceed by induction on $s \geq 0$, the base case being an identity.
Suppose then that $s \geq 1$.  For $\vh \in \Z^{s-1}$ and $h \in \Z$, we note that
$$
\Delta_{(\vh, h)} F(x) = \Delta_{\vh}\brac{ \E_{y, y' \in [M]} F_y(x)\overline{F_{y'}(x+h)}}
$$
Hence by the induction hypothesis 
\begin{multline*}
 \brac{N^{-s-1}\sum_h\sum_{\substack{\vh \\ (\vh,h) \in \mathcal{H}}}\abs{ \sum_x  \Delta_{(\vh,h)} F(x) e\bigbrac{\phi(\vh)x }}}^{2^s} \ll_s  \\
 \brac{N^{-2s}\sum_h  \sum_{\substack{\vh^{0}, \vh^{1}\\ (\vh^i,h) \in \mathcal{H}}}\abs{\sum_x \E_{y ,y'\in [M]} \Delta_{\vh^0-\vh^1} F_y(x)\overline{F_{y'}(x+h)} e\bigbrac{\phi(\vh^0;\vh^1;h)x }}}^2 ,
\end{multline*}
where
$$
\phi(\vh^0; \vh^1;h) := \sum_{\omega \in \set{0, 1}^{s-1}} (-1)^{|\omega|} \phi(\vh^{\omega},h).
$$
Letting $e(\psi(\vh^0;\vh^1;h))$ denote the phase of the inner absolute, we take the sum over $h$ inside and apply Cauchy--Schwarz to obtain
\begin{multline*}
\brac{\sum_{\vh^{0}, \vh^{1},x}  \E_{y ,y'\in [M]} \sum_{\substack{h\\ (\vh^i,h) \in \mathcal{H}}}\Delta_{\vh^0-\vh^1} F_y(x)\overline{F_{y'}(x+h)} e\bigbrac{\phi(\vh^0;\vh^1;h)x +\psi(\vh^0;\vh^1;h)}}^2
\\ \leq N^{2s-1} \sum_{\vh^{0}, \vh^{1}}  \sum_{\substack{h^0, h^1\\ (\vh^i,h^j) \in \mathcal{H}}}\\ \abs{\sum_x \E_{y \in [M]} \Delta_{\vh^0-\vh^1} F_y(x)\overline{F_{y}(x+h^0-h^1)} e\Bigbrac{\bigbrac{\phi(\vh^0;\vh^1;h^0)-\phi(\vh^0;\vh^1;h^1)}x}}.
\end{multline*}
The result follows.
\end{proof}

If $\phi(h_1, \dots, h_{s-1})$ is a function of $s-1$ variables we write
$
\phi(h_1, \dots, \hat{h}_i, \dots, h_{s}) := \phi(h_1, \dots, h_{i-1}, h_{i+1}, \dots, h_{s}).
$
We say that $\phi(h_1, \dots, h_s)$ is \emph{low rank} if there exist functions $\phi_i(h_1, \dots, h_{s-1})$ such that
$$
\phi(h_1, \dots, h_s) = \sum_{i=1}^s \phi_i(h_1, \dots, \hat{h}_i, \dots, h_s).
$$
From the definition of the Gowers norm together with the $U^2$-inverse theorem (Lemma \ref{U2 inverse}), one can show that largeness of the $U^{s+2}$-norm is equivalent to the existence of $\phi :\Z^s \to \T$ such that
$$
\sum_{h_1, \dots, h_s} \abs{\sum_x \Delta_h f(x)e(\phi(h)x)} \gg N^{s+1}.
$$
The following lemma says that if $\phi$ is low-rank, then the $U^{s+1}$-norm must also be large.
\begin{lemma}[Low rank correlation implies lower degree]\label{low rank lemma}
Let $f : \Z \to \C$ be a 1-bounded function with support in $[N]$.  Then for $\phi_1, \dots, \phi_m : \Z^{s-1} \to \T$ with $m \leq s$ we have
\begin{equation}\label{low rank ineq}
 \recip{N^{s+1}} \sum_{h_1, \dots, h_s}\abs{\sum_x \Delta_{h} f(x) e\brac{\sum_{i=1}^m\phi_i(h_1, \dots, \hat{h}_i, \dots, h_s)x}}\\ \ll_m  \brac{\frac{\norm{f}_{U^{s+1}}^{2^{s+1}}}{N^{s+2}}}^{2^{-m-1}}.
\end{equation}
 \end{lemma}

\begin{proof}
We proceed by induction on $m \geq 0$, the base case corresponding to the Cauchy--Schwarz inequality.  Suppose then that $m \geq 1$ and the result is true for smaller values of $m$.  Letting $e(\psi(h))$ denote the phase of the inner-most sum, the left-hand side of \eqref{low rank ineq} is equal to
\begin{multline*}
 \recip{N^{s+1}} \sum_{h_2, \dots, h_s, x}\Delta_{h_2, \dots, h_s}f(x)e\brac{\phi_1(h_2, \dots, h_s)}\sum_{h_1} \Delta_{h_2, \dots, h_s} \overline{f}(x+h_1) \\
 e\brac{\sum_{i=2}^m\phi_i(h_1, \dots, \hat{h}_i, \dots, h_s)x+ \psi(h_1, \dots, h_s)}.
\end{multline*}
By Cauchy--Schwarz, the square of this is at most
\begin{multline*}
 \recip{N^{s+2}} \sum_{h_2, \dots, h_s}\ \sum_{h_1, h_1'\in (-N, N)}\\\abs{\sum_x\Delta_{h_1-h_1',h_2, \dots, h_s} f(x) 
 e\brac{\sum_{i=2}^m\brac{\phi_i(h_1, \dots, \hat{h}_i, \dots, h_s)-\phi_i(h_1', \dots, \hat{h}_i, \dots, h_s)}x}}.
\end{multline*}
Taking a maximum over $h_1' \in (-N, N)$ and changing variables in $h_1$, the latter is at most an absolute constant times
\begin{multline*}
 \recip{N^{s+1}} \sum_{h_1, h_2, \dots, h_s}\Bigg|\sum_x\Delta_{h_1,h_2, \dots, h_s} f(x) \\
 e\brac{\sum_{i=2}^m\brac{\phi_i(h_1+h_1',h_2 \dots, \hat{h}_i, \dots, h_s)-\phi_i(h_1',h_2 \dots, \hat{h}_i, \dots, h_s)}x}\Bigg|.
\end{multline*}
This phase  has lower rank than the original, hence we may apply the induction hypothesis to yield the lemma.
\end{proof}

\begin{lemma}[Degree lowering]\label{degree lowering lemma}
There exists an absolute constant such that for $N \geq C(q/\delta)^{C}$ the following holds. Let $f_0, f_1 : \Z \to \C$ be 1-bounded functions with support in $[N]$ and define the dual 
$$
F(x) := \E_{y \in [M]} f_0(x-qy^2) f_1(x+y-qy^2).
$$
If, for $s \geq 3$, we have
$$
\sum_{u \in [q]}\norm{F}_{U^s(u + q \cdot \Z)}^{2^s} \geq \delta \sum_{u \in [q]}\norm{1_{[N]}}_{U^s(u + q \cdot \Z)}^{2^s},
$$
then
$$
\sum_{u \in [q]}\norm{F}_{U^{s-1}(u + q \cdot \Z)}^{2^{s-1}} \gg_s \delta^{4^{s+2}} \sum_{u \in [q]}\norm{1_{[N]}}_{U^{s-1}(u + q \cdot \Z)}^{2^{s-1}},
$$
\end{lemma}

\begin{proof}  Write $M := \floor{(N/q)^{1/2}}$.  
Given $u \in [q]$ let 
$
F_u(x) := F(u + qx)$, a function with support in the interval $[2N/q]$.  Applying the popularity principle, there exists a set of $\Omega(\delta q)$ residues $u \in [q]$ for which $\norm{F_u}_{U^s}^{2^s} \gg \delta (N/q)^{s+1}$.  Expanding the definition of the $U^s$-norm \eqref{Us def} we have
$$
\sum_{h_1, \dots, h_{s-2}} \norm{\Delta_{h_1, \dots, h_{s-2}}F_u}_{U^2}^4 \gg \delta (N/q)^{s+1}.
$$
Applying the $U^2$-inverse theorem (Lemma \ref{U2 inverse}), there exists $\mathcal{H} \subset (-2N/q, 2N/q)^{s-2}$ of size $|\mathcal{H}| \gg \delta (N/q)^{s-2}$ and a function $\phi : \Z^{s-2} \to \T$ such that for every $\vh \in \mathcal{H}$ we have
\begin{equation}\label{phi correlation}
\abs{\sum_x \Delta_{\vh} F_u(x) e\bigbrac{\phi(\vh)x} }\gg \delta N/q.
\end{equation}

Set $T := \ceil{C\delta^{-1}N/q}$, with $C$ an absolute constant taken sufficiently large to ensure that, on rounding $\phi(\vh)$ to the nearest fraction of the form $t/T$, the validity of \eqref{phi correlation} remains.  Summing over $\vh \in \mathcal{H}$ and applying Lemma \ref{dual difference interchange}, we deduce that 
\begin{multline*}
 \sum_{\vh^{0}, \vh^{1}\in \mathcal{H}} \abs{\sum_x\E_{y \in [M]} \Delta_{\vh^0-\vh^1} f_0(u+qx-qy^2)\Delta_{\vh^0-\vh^1}f_1(u+qx+y-qy^2)}\\
  e\bigbrac{\phi(\vh^0; \vh^1)x} \gg_s \delta^{2^{s-1}} (N/q)^{2s-1}.
\end{multline*}
Applying the pigeon-hole and popularity principle, there exists $\mathcal{H}' \subset \mathcal{H}$ of size $\Omega_s(\delta^{2^{s-1}} (N/q)^{s-2})$ and $\vh^{1} \in \mathcal{H}$ such that for every $\vh^0 \in \mathcal{H}'$ we have
\begin{multline*}
 \abs{\sum_x \sum_{y \in [M]} \Delta_{\vh^0-\vh^1} f_0(u+qx-qy^2)\Delta_{\vh^0-\vh^1}f_1(u+qx+y-qy^2)
  e\bigbrac{\phi(\vh^0, \vh^1)x }}\\ \gg \delta^{2^{s-1}} MN/q.
\end{multline*}

By Lemma \ref{lem6.2}, for each $\vh^0\in \mathcal{H}'$ there exists $q' \ll \delta^{-2^{s+1}}$ such that $$\norm{q'q^2\phi(\vh^0, \vh^1)}\ll \delta^{-2^{s}\times 7} q^3/N$$  Notice that $\phi(\vh^0, \vh^1)$ is an element of the additive group $\set{t/T : t \in [T] }\subset \T$.  Moreover, for any $Q_i$ we have the inclusion 
$$
\set{\alpha \in \T : \exists q' \leq Q_1 \text{ with } \norm{q'q^2\alpha}\leq Q_2 q^3/N} \subset 
\bigcup_{\substack{1 \leq a \leq q \leq Q_1\\ \hcf(a, q) = 1}}  \sqbrac{ \frac{a}{q'q^2} - \frac{Q_2}{N}, \frac{a}{q'q^2} + \frac{Q_2 }{N}}.
$$
By a volume packing argument, the number of $t/T$ lying in this union of intervals is at most $O\brac{Q_1^2(1+\tfrac{Q_2T}{N})}$.  It therefore follows from the pigeon-hole principle that there exists $\mathcal{H}''\subset \mathcal{H}'$ of size $\Omega\brac{\delta^{2^{s+3}+1-2^s} (N/q)^{s-2}}$ and $t_0 \in [T]$ such that for any $\vh^0\in \mathcal{H}''$ we have $\phi(\vh^0, \vh^1) = t_0/T$.  In particular, when restricted to the set $\mathcal{H}''$, the function $\phi$ satisfies
$$
\phi(\vh^0) =  t_0/T - \sum_{\omega\in \set{0,1}^s\setminus\set{0}} (-1)^{|\omega|} \phi(\vh^{\omega}).
$$
The right-hand side of this identity is clearly \emph{low rank} according to the terminology preceding Lemma \ref{low rank lemma}.

Summing over $\vh\in \mathcal{H}''$ in \eqref{phi correlation}, we deduce the existence of a  low rank function $\psi: \Z^{s-2} \to \T$ such that
$$
\sum_{\vh}\abs{\sum_x F_u(x) e\bigbrac{\psi(\vh)x} }\gg \delta^{2^{s+3}+1-2^s} (N/q)^{s-1}.
$$
Employing Lemma \ref{low rank lemma} then gives
$$
\norm{F_u}_{U^{s-1}}^{2^{s-1}} \gg \delta^{(2^{s+3}+1-2^s)2^{s+1}}(N/q)^s.
$$
Summing over permissible $u$, then extending to the full sum over $ u\in [q]$ by positivity, we obtain the bound claimed in the lemma.
\end{proof}

\section{Proof of the cut norm inverse theorem}\label{inverse theorem section}

In this section we complete our proof of Theorem \ref{partial inverse theorem}. We first show how the dual function is controlled by the $U^5$-norm, and hence by the degree lowering of \S\ref{sec6}, the dual is controlled by the $U^1$-norm.

The following can be found in the discussion following \cite[Proposition 3.6]{GowersDecompositions}. Although the statement therein is for norms, and not seminorms, one can check that the (simple) argument remains valid in this greater generality\footnote{On occasion the relevant results in \cite{GowersDecompositions} appear to assume that unit balls are \emph{bounded} (if we take the definition of \emph{convex body} to be a compact convex set with non-empty interior), which may not be true for the unit ball of a seminorm.  However, the boundedness assumption is not necessary in the pertinent proofs. Moreover, one could quotient by the norm zero set to obtain a genuine norm.}. %\comment{double check!}
\begin{lemma}\label{HB decomp}
Let $\|\cdot\|$ be a seminorm on the space of complex-valued functions supported on $[N]$.  For any such  function $f$ and $\eps > 0$ there exists a decomposition $f=f_{str}+f_{unf}$ such that 
$$
\norm{f_{str}}^* \leq \eps^{-1}\norm{f}_{2} \quad \text{and} \quad \norm{f_{unf}}\leq \eps \norm{f}_{2}.
$$
\end{lemma}
\begin{lemma}[$U^5$-control of the dual]\label{dual U5 control} There exists an absolute constant $C$ such that for $N \geq Cq\delta^{-C}$ the following holds.
Let $g_0, g_1, f : \Z \to \C$ be 1-bounded functions, each with support in $[N]$.  Suppose that
$$
\abs{\Lambda_{q, N}(g_0, g_1, f)} \geq \delta \Lambda_{q, N}(1_{[N]}).
$$
Then, on defining the dual 
\begin{equation}\label{dual defn}
G(x) := \E_{y \in [M]} g_0(x-qy^2) g_1(x+y-qy^2),
\end{equation}
we have
$$
\sum_{u \in [q]}\norm{G}_{U^5(u + q \cdot \Z)}^{2^5} \gg \delta^{2^{26}} \sum_{u \in [q]}\norm{1_{[N]}}_{U^5(u + q \cdot \Z)}^{2^5}.
$$
\end{lemma}

\begin{proof}
Applying Lemma \ref{HB decomp} to $f$ with  $\norm{\cdot} := \norm{\cdot}^\sharp_q$ as defined in \eqref{sharp q norm eq} and $\eps := \trecip{2}\delta \Lambda_{q, N}(1_{[N]})N^{-1/2}$, we  deduce that 
\begin{multline*}
|\Lambda_{q, N}(g_0, g_1,f_{str})|    \geq \delta\Lambda_{q, N}(1_{[N]}) -  |\Lambda_{q, N}(g_0, g_1, f_{unf})|\\
 \geq \delta\Lambda_{q, N}(1_{[N]}) -  \norm{f_{unf}}_{q, N}^\sharp \geq \tfrac{1}{2} \delta\Lambda_{q, N}(1_{[N]}).
\end{multline*}

We note that our lower bound assumption on $N$ implies that $\Lambda_{q, N}\brac{1_{[N]}}\gg 1$.  Hence the dual inequality \eqref{dual ineq} gives
$$
 \delta  \ll N^{-1}|\ang{f_{str}, G}| \ll \delta^{-1} \norm{G}^\sharp_q.
$$
Invoking Theorem \ref{global U5} yields the result.
\end{proof}

Taken together, the work in \S\S\ref{pet}--\ref{sec6} gives the following.

\begin{proof}[Proof of Theorem \ref{partial inverse theorem}]
Applying Lemma \ref{dual U5 control}, we deduce that 
$$
\sum_{u \in [q]}\norm{G}_{U^5(u + q \cdot \Z)}^{2^5} \gg \delta^{2^{26}} \sum_{u \in [q]}\norm{1_{[N]}}_{U^5(u + q \cdot \Z)}^{2^5},
$$
where $G$ is defined as in \eqref{dual defn}.

We now apply Lemma \ref{degree lowering lemma} three times.  The first application gives 
$$
\sum_{u \in [q]}\norm{G}_{U^4(u + q \cdot \Z)}^{2^4} \gg \delta^{2^{40}} \sum_{u \in [q]}\norm{1_{[N]}}_{U^4(u + q \cdot \Z)}^{2^4},
$$
a second replaces $U^4$ with $U^3$ at the cost of replacing $\delta^{2^{40}}$ with $\delta^{2^{52}}$. With a final application, we obtain
$$
\sum_{u \in [q]}\norm{G}_{U^2(u + q \cdot \Z)}^{4} \gg \delta^{2^{62}} \sum_{u \in [q]}\norm{1_{[N]}}_{U^2(u + q \cdot \Z)}^{4}.
$$

Let $\eta:= \delta^{2^{62}}$.  By the popularity principle, there are at least $\Omega(\eta q)$ values of $u \in [q]$ for which $\norm{G}_{U^2(u + q \cdot \Z)}^{4} \gg \eta \norm{1_{[N]}}_{U^2(u + q \cdot \Z)}^{4}$.  The inverse theorem for the $U^2$-norm then gives the existence of $\phi(u) \in \T$ for which 
\begin{equation}\label{u correlation}
\abs{\sum_x G(u + qx) e(\phi(u) x)} \gg \eta^{1/2} N/q.
\end{equation}
Set $T := \ceil{C\eta^{-1/2}N/q}$, with $C$ an absolute constant taken sufficiently large to ensure that, on rounding $\phi(u)$ to the nearest fraction of the form $t/T$, the inequality \eqref{u correlation} remains valid.  

By Lemma \ref{lem6.2}, for each $u$ satisfying \eqref{u correlation}, there exists a positive integer $q'\ll\eta^2$ such that $\|q'q^2\phi(h)\|\ll\eta^{-7}q^3/N$.  By a volume packing argument similar to that given in the proof of Lemma \ref{degree lowering lemma}, the function $\phi$ is constant on a proportion of at least $\Omega\bigbrac{\eta^{11}}$ of the residues $u \in [q]$ satisfying \eqref{u correlation}.  Summing over these $u$, then extending the sum to all of $[q]$, we deduce the existence of $\alpha \in \T$ and $q'\ll\eta^{-2}$ such that $\|q'q^2\alpha \|\ll\eta^{-7}q^3/N$ and
\begin{equation}\label{Gu discrep}
\sum_{u \in [q]} \abs{\sum_x G(u + qx) e(\alpha x)} \gg \eta^{12} N.
\end{equation}

Expanding the dual function, there is a 1-bounded function $\psi(u\bmod q)$ such that the left-hand side of the above is equal to
\begin{multline}\label{hard correlation}
 \sum_{u \in [q]} \psi(u\bmod q)\sum_{x\equiv u(q)} \E_{y \in [M]}g_0(x-qy^2)g_1(x+y-qy^2) e(\alpha x/q)\\
 = \sum_x g_0(x)\psi(x\bmod q)e(\alpha x/q) \E_{y \in [M]}g_1(x+y)e(\alpha y^2).
\end{multline}

Let us first suppose that $f = g_0$, we deal with the case $f=g_1$ shortly.  Setting 
$$
\phi(x) := \psi(x\bmod q)e(\alpha x/q) \E_{y \in [M]}g_1(x+y)e(\alpha y^2),
$$
we have  $\ang{f, \overline{\phi}} \gg \eta^{12} N$. Our aim is to show that $\phi$ can be approximated by a local function of the type claimed in the lemma.

We begin by removing the phase from the expectation over $[M]$, at the cost of passing to shorter progressions.  Let $M'\leq M/q'q^2$ be a quantity to be determined.  If $y \in [M']$ then for any $m \in [-M, M]\cap \Z$ we have
\begin{equation}\label{quadratic phase periodicity}
\abs{e(\alpha (m+q'q^2y)^2) - e(\alpha m^2)} 
 \ll \norm{ \alpha\brac{2mq'q^2y + (q'q^2y)^2}}
 \ll q'q^{4} \eta^{ -7} M'/M .
\end{equation}
Hence, partitioning $\Z$ into progressions $P$ of common difference $q'q^2$ and length $M'$, there exist phases $\omega_P$ such that for any $x \in \Z$ we have
\begin{equation}\label{omegaP}
\abs{\E_{y\in [M]} g_1(x+y)e(\alpha y^2) -M^{-1}\sum_P\omega_P\sum_{y \in [M]\cap P} g_1(x+y)} \ll q'q^{4} \eta^{ -7} M'/M .
\end{equation}
Notice that there are at most $O(M/M')$ progressions $P$ such that $P\cap [M] \neq \emptyset$ (since we are assuming $M' \leq M/q'q^2$).

Next we show how the phase $e(\alpha x/q)$ is approximately periodic.  Suppose that $z \in  [M'']$, with $M''\leq M'/q$ to be determined. Then for any $x \in \Z$ we have
$$
\abs{e\brac{\alpha (x + q'q^3 z)/q}-e\brac{\alpha x}} \ll \norm{\alpha q'q^2} M'' \ll \eta^{-7}q^3M''/ N
$$
and by a boundary estimate
$$
\abs{\sum_{y \in [M]\cap P}g_1(x+q'q^3z + y) - \sum_{y \in [M]\cap P}g_1(x + y)}  \ll qM''.
$$
It then follows from a telescoping identity that for all $x \in \Z$ and $z \in [M'']$ we have
\begin{align*}
\abs{\phi(x+q'q^3z) - \phi(x)} & \ll  \frac{\eta^{-7}q^3M''}{N} +\frac{\eta^{-7}q'q^{4}  M'}{M} + \frac{qM''}{M}\sum_{\substack{P \\ P\cap [M] \neq \emptyset}} 1\\
& \ll \frac{\eta^{-7}q'q^{4}  M'}{M} + \frac{qM''}{M'}.
\end{align*}
Taking $M' := c\eta^{19} M/q'q^4$ and $M'' := c\eta^{12} M'/q$ for a sufficiently small absolute constant $c>0$ we have
\begin{equation}\label{phi invariance}
\abs{\phi(x+q'q^3z) - \phi(x)} \leq \eta^{12}/C \quad \text{for all }x\in \Z \text{ and } z \in [M''].
\end{equation}

Partitioning $\Z$ into translates $T$ of $q'q^3\cdot [M'']$ we deduce that
$$
\sum_T \biggabs{\sum_{x \in T} f(x)} \gg \eta^{12} N.
$$
Write $\chi(x)$ for the phase of the inner sum when $x\in T$. Then $\chi$ is a 1-bounded local function of modulus $q'q^3$ and resolution $\Omega\brac{(\delta/q)^{O(1)} M}$ satisfying
$$
\sum_x f(x)\overline{\chi(x)} \gg \delta^{2^{66}}N,
$$
as required.

Next we give the argument for when $f = g_1$. Returning to \eqref{hard correlation} we have
$$
\sum_x \abs{\E_{y \in [M]} f(x+y)e(\alpha y^2)} \gg \eta^{12} N.
$$
Utilising \eqref{quadratic phase periodicity} and \eqref{omegaP}, we may partition $\Z$ into progressions $P$ of common difference $q'q^2$ and length $M':=c\eta^{19} M/q'q^4$ such that 
\begin{equation*}
\sum_x\sum_P\abs{\E_{y \in [M]\cap P} f(x+y)} \gg \eta^{12}N .
\end{equation*}
Since $O(M/M')$ of the $P$ intersect $[M]$, the pigeon-hole principle gives $P' := P \cap [M]$ such that
\begin{equation*}
\sum_x\abs{\sum_{y \in P'} f(x+y)} \gg \eta^{12}N M'.
\end{equation*}
In particular $|P'| \gg \eta^{12} M' \gg (q/\delta)^C M$.

Partitioning $\Z$ into translates of $P'$ of the form
$$
\Z = \bigsqcup_i (a_i + P'),
$$
the pigeon-hole principle gives $z \in P'$ such that
$$
\sum_{i} \abs{\sum_{y \in P'} f(a_i+y+z)} \gg \eta^{12} N.
$$
Writing $\chi(x)$ for the phase of the inner sum when $x \in a_i + P$ one sees that $\chi$ is a local function of resolution $\gg (q/\delta)^C M$ and modulus $q'q^2$ which satisfies $\ang{f, \chi} \gg \eta^{12} N$.  The proof is complete on noting that a local function of modulus $q'q^2$ is also a local function of modulus $q'q^3$.
\end{proof}

\appendix

\section{Basic theory of the Gowers norms}

\begin{lemma}[Inverse theorem for the $U^2$-norm]\label{U2 inverse}
Let $f:\Z\to\C$ be a $1$-bounded function with support in $[N]$. 
Then there exists $\alpha \in \T$ such that
\[
\|f\|_{U^2}^4 \leq N\left|\sum_{x}f(x)e(\alpha x)\right|^2.
\]
\end{lemma}
\begin{proof}
Using the definition of the Fourier transform \eqref{Fourier transform}, together with orthogonality of additive characters, we have
$$
\norm{f}_{U^2}^4 = \int_{\T} \bigabs{\hat{f}(\alpha)}^4\intd\alpha \leq \bignorm{\hat{f}}_\infty^2 \int_{\T} \bigabs{\hat{f}(\alpha)}^2\intd\alpha \leq \bignorm{\hat{f}}_\infty^2 N.
$$
\end{proof}
%\begin{corollary}
%Let $r\in [N]$ and $f:\Z\to\C$ be a $1$-bounded function supported on $[N]$. If
%\[
%\|f\|_{U^2(r \cdot \Z)}\geq\delta \|1_{[N]}\|_{U^2(r\cdot \Z)},
%\]
%then there exists a real number $\alpha$ such that
%\[
%\left|\sum_xf(rx)e(\alpha x)\right|\gg\delta^2 (N/r).
%\]
%\end{corollary}
%\begin{proof}
%Define a function $\tilde{f}:\Z\to\C$ by $\tilde{f}(x)=f(rx)$, so that $\tilde{f}$ is supported on $[N/r]$ and $\|f\|_{U^2(r\cdot \Z)}=\|\tilde{f}\|_{U^2}$. Apply the previous lemma to $\tilde{f}$.
%\end{proof}
For each $\omega\in\{0,1\}^s$, let $f_\omega:\Z\to\C$ be a function with finite support. Then we define the \emph{Gowers inner product} by
$$
[f_\omega]_{U^s} := \sum_{x, h_1, \dots, h_s} \prod_{\omega \in \set{0, 1}^s} \mathcal{C}^{|\omega|}f_\omega(x + \omega \cdot h).
$$
Here $\mathcal{C}$ denotes the operation of complex conjugation. Notice that $[f]_{U^s} = \norm{f}_{U^s}^{2^s}$.
\begin{lemma}[Gowers--Cauchy--Schwarz]
For each $\omega\in\{0,1\}^s$, let $f_\omega:\Z\to\C$ be a function with finite support. Then we have
\[
[f_\omega]_{U^s}\leq \prod_{\omega\in\{0,1\}^s}\|f_\omega\|_{U^s}.
\]
\end{lemma}

\begin{proof}
See \cite[Exercise 1.3.19]{TaoHigher}.
\end{proof}

\begin{lemma}[Phase invariance for $s \geq 2$]\label{phase invariance}
Let $L \in \R[x, h_1, \dots, h_s]$ be a linear form, with $s\geq 2$ and let $f : \Z \to \C$.  Then 
$$
\biggabs{\sum_{x, h_1, \dots, h_s} \Delta_{h_1, \dots, h_s} f(x) e(L(x, h_1, \dots, h_s))} \leq \norm{f}_{U^s}^{2^s}.
$$
\end{lemma}

\begin{proof}
The linear form may be written as
$$
L(x, h_1, \dots, h_s) = \alpha x + \beta_1(x+h_1) + \dots + \beta_s(x+h_s),
$$
for some real $\alpha$ and $\beta_i$.  Write $f_0(x) := f(x)e(\alpha x)$, $f_{e_i}(x) := f(x) e(-\beta_i x)$ for $i = 1, \dots, s$, and for $\omega \in \set{0,1}^s\setminus \set{0, e_1, \dots, e_s}$ set $f_\omega := f$. Then by Gowers--Cauchy--Schwarz we have
$$
\biggabs{\sum_{x, h_1, \dots, h_s} \Delta_{h_1, \dots, h_s} f(x) e(L(x, h_1, \dots, h_s))} \leq \prod_{\omega} \norm{f_\omega}.
$$
It therefore suffice to prove that for a phase function $e_\alpha : x \mapsto e(\alpha x)$
$
\norm{fe_\alpha}_{U^s} = \norm{f}_{U^s}.
$
The latter follows on observing that 
$$
\Delta_{h_1, \dots, h_s} (fe_\alpha) = \brac{\Delta_{h_1, \dots, h_s}f}\brac{\Delta_{h_1, \dots, h_s}e_\alpha},
$$
and for any $x, h_1, \dots, h_s$ with $s \geq 2$ we have
$
\Delta_{h_1, \dots, h_s}e_\alpha(x) = 1.
$
\end{proof}

\begin{lemma}[Box Cauchy--Schwarz]\label{box cauchy}
Let $\mu_1, \mu_2, \mu_3$ be probability measures on $\Z$ with the discrete sigma algebra.  If $F_1, F_2, F_3$ are 1-bounded function on $\Z^2$ and $F$ is a 1-bounded function on $\Z^3$ then
\begin{multline*}
\abs{\sum_{x \in \Z^3} F_1(x_2, x_3)F_2(x_1, x_3) F_3(x_1, x_2) F(x)\underline{\mu}(x)}^8\\ \leq \sum_{x^0, x^1\in \Z^3} \prod_{\omega \in \set{0,1}^3} \mathcal{C}^{|\omega|} F(x_1^{\omega_1}, x_2^{\omega_2},x_3^{\omega_3})\mu_1(x_1^{0})\mu_1(x_1^{1})\mu_2(x_2^{0})\mu_2(x_2^{1})\mu_3(x_3^{0})\mu_3(x_3^{1}).\end{multline*}

\end{lemma}

%\renewcommand{\refname}{\normalsize References}
%\bibliographystyle{plain}
%\bibliography{bib}
{
%\footnotesize%\small%\scriptsize
 %\bibliographystyle{habbrv}
% \bibliographystyle{abbrvnat-noDOIorURL}
 %\bibliographystyle{abbrv}
  %\bibliographystyle{acm}
 % \bibliographystyle{hsiam}  
  \bibliographystyle{alphaabbr}
 \bibliography{../../../../SeanBib.bib}
 }

\end{document}